\documentclass[a4paper, 10pt]{amsart}

\usepackage[utf8]{inputenc}

\usepackage{enumerate}
\usepackage{amsthm}
\usepackage{amsmath}
\usepackage{amssymb}
\usepackage{scalerel}
\usepackage{mathtools}
\usepackage{bm}
\usepackage{tikz}
\usepackage{IEEEtrantools}
\usepackage{hyperref}
\usepackage{graphicx}
\usepackage{emptypage}
\usepackage{color}
\usepackage{xifthen}
\usepackage{thmtools}
\usepackage{thm-restate}

\newtheorem*{thm*}{Theorem}
\newtheorem{thm}{Theorem}[section]

\newtheorem{lemma}[thm]{Lemma}
\newtheorem{prop}[thm]{Proposition}
\newtheorem{corollary}[thm]{Corollary}

\theoremstyle{definition}
\newtheorem{mydef}{Definition}

\theoremstyle{remark}
\newtheorem{remark}{Remark}
\newtheorem{example}{Example}

\numberwithin{equation}{section}

\newenvironment{ieee}[1]{\begin{IEEEeqnarray}{#1}}{\end{IEEEeqnarray}\ignorespacesafterend}
\newenvironment{ieee*}[1]{\begin{IEEEeqnarray*}{#1}}{\end{IEEEeqnarray*}\ignorespacesafterend}

\newcommand{\ie}{\emph{i.e.}}
\newcommand{\hsmallspace}{\hspace{0.5cm}}

\newcommand{\R}{\mathbb{R}}

\newcommand{\N}{\mathbb{N}}

\newcommand{\prob}{\mathcal{P}}
\newcommand{\Leb}{\mathcal{L}}
\newcommand{\energy}[1]{\mathcal{E}^{#1}}
\newcommand{\eps}{\varepsilon}
\newcommand{\wconv}{\rightharpoonup}

\newcommand{\Rdn}{(\R^d)^N}
\newcommand{\de}{\mkern2.5mu\mathrm{d}}
\newcommand{\dX}[1][]{%
	\ifthenelse{\isempty{#1}}{\de X}{\de \hat{X}_{#1}}%
}
\newcommand{\dY}[1][]{%
	\ifthenelse{\isempty{#1}}{\de Y}{\de \hat{Y}_{#1}}%
}
\newcommand{\proj}[2]{#1{\mkern-1.5mu\downharpoonright}_{\mkern-.5mu#2}}
\newcommand{\eqdef}{\vcentcolon=}

\newcommand{\gra}[1]{\left \{ #1 \right \}}
\newcommand{\st}{\left | \right.}
\newcommand{\abs}[1]{\left| #1 \right|}
\newcommand{\norm}[2][]{\left\lVert {#2} \right\rVert_{#1}}

\title{Smoothing operators in multi-marginal Optimal Transport}
\author{Ugo Bindini}
\address{Scuola Normale Superiore \\ Piazza dei Cavalieri, 7 \\ 56126 Pisa - ITALY}
\email{ugo.bindini@sns.it}
\date{\today}

\begin{document}

\maketitle

\begin{abstract}
	Given $N$ absolutely continuous probabilities $\rho_1, \dotsc, \rho_N$ over $\R^d$ which have Sobolev regularity, and given a transport plan $P$ with marginals $\rho_1, \dotsc, \rho_N$, we provide a universal technique to approximate $P$ with Sobolev regular transport plans with the same marginals. Moreover, we prove a sharp control of the energy and some continuity properties of the approximating family.
\end{abstract}

\section{Introduction}

We consider a multi-marginal Optimal Transport problem on the Euclidean space: given $N$ Borel probability measures $\rho_1, \dotsc, \rho_N \in \prob(\R^d)$, and given a cost function $c : \Rdn \to \R$, the goal is to find
\begin{equation} \label{otproblem}
	\min_{P} \int c(x_1, \dotsc, x_N) \de P(x_1, \dotsc, x_N)
\end{equation}
under the constraint
\[  P \in \Pi(\rho_1, \dotsc, \rho_N) \eqdef \gra{P \in \prob \left( \Rdn \right) \st \pi_\#^j P = \rho_j \ \forall j = 1, \dotsc, N}. \]
Here $\pi^j \colon \Rdn \to \R^d$ denotes the projection onto the $j$-th coordinate, \ie, $\pi^j(x_1, \dotsc, x_N) = x_j$.

When $N = 2$, the classical Kantorovich formulation of the Optimal Transport problem is recovered; however, many characteristics of the multi-marginal problem are different from the classical one. For a good survey on both cases see for instance \cite{ambrosio2013user, pass2015multi}. 

In this work we want to investigate the properties of the space $\Pi(\rho_1, \dotsc, \rho_N)$ when the measures $\rho_1, \dotsc, \rho_N$ share some regularity --- in particular, we are interested in the case when the marginals have a Sobolev-type regularity, as clarified in the following

\begin{mydef} \label{p-definition}
	If $p \geq 1$, we say that a probability measure $\mu \in \prob(\R^m)$ is $W^{1,p}$\emph{-regular} if $\mu$ is absolutely continuous with respect to the Lebesgue measure $\Leb^m$, and
	\[ \left( \frac{\de \mu}{\de \Leb^m} \right)^{\frac{1}{p}} \in W^{1,p}(\R^m). \]	
	In other words, $\mu$ is $W^{1,p}$-regular if there exists $f \in W^{1,p}(\R^m)$, $f \geq 0$, such that
	\[ \frac{\de \mu}{\de \Leb^m} = f^p. \]
\end{mydef}

We will denote by $\prob^{1,p}(\R^m)$ the space of $W^{1,p}$-regular probability measures. This definition arises naturally in the setting of Density Functional Theory as a generalization of the one given by Lieb in \cite{lieb2002density} for $p = 2$. In what follows, when we say that a measure is regular we will mean that it is $W^{1,p}$-regular for some fixed $p$. After giving some basic notation and results in \autoref{section-notation}, we study in \autoref{section-regular-measures} the properties of regular measures, stressing in particular the relation between a measure and its marginals.

Even when the marginals $\rho_1, \dotsc, \rho_N$ are regular, the optimal plan in \eqref{otproblem} may be singular; it is well known, for instance, that in the case $N=2$, under suitable hypotheses, the optimal plan is concentrated on a graph. On the other hand, for many applications, and in particular when dealing with $\Gamma$-convergence, it may be useful to construct regular transport plans which are ``close'' to a given optimal one (see for instance \cite{bindini2017optimal, cotar2018smoothing, lewin2018semi}). With this in mind, in Sections \ref{section-definition}--\ref{section-mu-continuity}, we address the following

\medskip
\textbf{Problem: } \label{main-problem} \textit{Given $\rho_1, \dotsc, \rho_N\in \prob^{1,p}(\R^d)$, and given $\mu \in \Pi(\rho_1, \dotsc, \rho_N)$, find a family $\left( \mu^\eps \right)_{\eps > 0}$ such that:
\begin{enumerate}[(i)]
	\item $\mu^\eps \in \Pi(\rho_1, \dotsc, \rho_N);$
	\item $\mu^\eps \in \prob^{1,p} \left( \Rdn \right);$
	\item $\mu^\eps \to \mu$ as $\eps \to 0$ (for a suitable notion of convergence).
\end{enumerate}}
\medskip

In other words, we search for $W^{1,p}$-regular multi-marginal transport plans with marginals $\rho_1, \dotsc, \rho_N$ which approximate a (non regular) transport plan $\mu$. Since in general $\mu$ could be no more regular than a measure, the natural topology for (iii) is the tight convergence of probability measures, \ie, weak convergence in duality with $C_b\left( \Rdn \right)$ (continuous and bounded functions).

Notice that, if $\mu$ is optimal in \eqref{otproblem}, and the cost $c$ is upper semi-continuous and bounded from above, combining (iii) and the Portmanteau's Theorem we get
\[ \lim_{\eps \to 0} \int c(X) \de \mu^\eps(X) = \int c(X) \de \mu(X), \]
whence we may say that $\mu^\eps$ is ``almost'' optimal for small $\eps$.

This problem has already been treated in C. Cotar, G. Friesecke and C. Kl\"{u}ppelberg in \cite{cotar2013density, cotar2018smoothing} and solved with a different construction for $p=2$. Our technique was introduced in collaboration with L. De Pascale in \cite{bindini2016gamma} and later used in \cite{bindini2017optimal} for studying the semiclassical limit in Density Functional Theory. Recently, our construction was extended to mixed states by M. Lewin in \cite{lewin2018semi}. In the present work we give a systematic presentation of the results for general $p\geq 1$, and we are also able to obtain sharp energy estimates (\autoref{energy-estimate} and \autoref{p-energy-estimate}) and a strong $W^{1,p}$-continuity property (\autoref{mu-continuity}). The latter, in particular, turns out to be a very useful tool in order to study the properties of the mapping between a transport plan and its marginals. We will use it, in a forthcoming work in preparation with L. De Pascale, to show that the map which sends a symmetric wave-function to its marginal is open, partially answering to a conjecture posed by Lieb in \cite[Question 2]{lieb2002density}.

Finally, we want to point out that the definition of the smoothing operator (Section \ref{section-definition}), which we give in the case of Sobolev spaces due to physical interest, works in the same way for other classes of absolutely continuous measures, \emph{e.g.}, measures with $C^{k,\alpha}$ density, with analogous regularity and continuity results.


\section{Notation and preliminary results} \label{section-notation}

We will denote by $\R^+$ the open interval $(0, +\infty)$. We recall the following elementary inequalities, valid for any $a,b \geq 0$:
\begin{ieee}{rll}
	\abs{a^p-b^p} &{} \leq \abs{a-b} \abs{a+b}^{p-1} \hsmallspace &1 \leq p < \infty \label{p-inequality} \\
	\abs{a^\gamma - b^\gamma} &{} \leq \abs{a-b}^\gamma \hsmallspace &0 < \gamma \leq 1. \label{gamma-inequality}
\end{ieee}

Given $\mu \in \prob\left( \Rdn \right)$, we denote its marginals by $\proj{\mu}{j} \eqdef \pi^j_\# \mu$, for $j=1, \dotsc, N$. If $f \colon \Rdn \to \R$, and $1 \leq j \leq N$, we denote by
\[ \int f(X) \dX[j] \eqdef \int f(x_1,\dotsc, x_N) \de x_1 \dotsm \de \hat{x}_j \dotsm \de x_N\]
the integral of $f$ with respect to all the variables except $x_j$. This is a function of the variable $x_j$.

When $f \in W^{1,p}(\R^m)$, we will denote by
\begin{equation} \label{abs-gradient-convention}
	\abs{\nabla f} \eqdef \left( \sum_{j = 1}^m \abs{\partial_{x_j} f}^p \right)^\frac{1}{p},
\end{equation}
\ie, when computing the norm of a gradient we take on $\R^m$ the $p$-th norm.

We say that a sequence of probability measures $\gra{\mu_k} \subseteq \prob(\R^m)$ weakly converges to $\mu \in \prob(\R^m)$, denoted $\mu_k \wconv \mu$, if for every $\phi \in C_b(\R^m)$ 
\[ \lim_{k \to \infty} \int \phi \de \mu_k = \int \phi \de \mu. \]

A family of measures $\mathcal{M} \subseteq \prob(\R^m)$ is said to be \emph{tight} if for every $\delta > 0$ there exists $K \subseteq \R^m$ compact such that $\mu(K) \geq 1 - \delta$ for every $\mu \in \mathcal{M}$.

Finally we recall the following classical results.

\begin{thm}[Prokhorov's theorem] \label{prokhorov} A family $\mathcal{M} \subseteq \prob(\R^m)$ is tight if and only if for every sequence $\gra{\mu_k} \subseteq \mathcal{M}$ there exists a subsequence $\gra{\mu_{n_k}}$ and $\mu \in \prob(\R^m)$ with $\mu_{n_k} \wconv \mu$.
\end{thm}

\begin{thm}[Generalized Lebesgue's dominated convergence theorem] \label{generalized-lebesgue} Let $\gra{f_n}_{n\in \N}$ and $\gra{g_n}_{n\in \N}$ be Lebesgue measurable functions, with $g_n \geq 0$. Suppose that:
	\begin{enumerate}[(i)]
		\item $\abs{f_n(x)} \leq g_n(x)$ for all $n \in \N$, for almost every $x$;
		\item $\gra{f_n}$ converges pointwise almost everywhere to $f$ and $\gra{g_n}$ converges pointwise almost everywhere to $g$;
		\item
		\[
		\lim_{n \to \infty} \int g_n = \int g. 
		\]
	\end{enumerate}
	
	Then $f$ is Lebesgue integrable on $E$ and
	\[
	\lim_{n \to \infty} \int f_n = \int f. 
	\]
\end{thm}

\subsection{Roots and powers of non-negative Sobolev functions}

The following Propositions will be useful later in order to have an expression for the weak derivatives of $p$-th powers and $p$-th roots of non-negative Sobolev functions.

\begin{prop} \label{p-root}
	Let $p > 1$. If $u \in W^{1,p}(\R^m)$, $u \geq 0$, then $u^p \in W^{1,1}(\R^m)$, and $\nabla u^p = p u^{p-1} \nabla u$.
	
	Viceversa, let $u \in W^{1,1}(\R^m)$, $u \geq 0$, such that
	\begin{equation} \label{finite-integral}
		\int u^{1-p} \abs{\nabla u}^p < \infty.
	\end{equation}
	
	Then $u^\frac{1}{p} \in W^{1,p}(\R^m)$, and $\nabla u^\frac{1}{p} = \frac{1}{p} u^\frac{1-p}{p} \nabla u$. 
\end{prop}

\begin{proof}
	If $u \in W^{1,p}(\R^m)$ clearly $u^p \in L^1(\R^m)$, and viceversa if $u \in W^{1,1}(\R^m)$ then $u^\frac{1}{p} \in L^p(\R^m)$. Let $u_n \in C^\infty(\R^m) \cap W^{1,p}(\R^m)$ such that $u_n \to u$ in $W^{1,p}(\R^m)$. Then by the H\"{o}lder inequality with exponents $p$ and $\frac{p}{p-1}$
	\begin{ieee*}{rCl}
		\int \abs{u_n^{p-1} \nabla u_n - u^{p-1} \nabla u} & \leq & \int u_n^{p-1} \abs{\nabla u_n - \nabla u} + \int \abs{\nabla u} \abs{u_n^{p-1} - u^{p-1}} \\
		& = & \norm[p]{u_n}^{p-1} \norm[p]{\nabla u - \nabla u_n} + \norm[p]{\nabla u} \norm[p]{\abs{u_n^{p-1} - u^{p-1}}^\frac{1}{p-1}}^{p-1}.
	\end{ieee*}
	
	If $p \geq 2$ we use \eqref{p-inequality} and the H\"{o}lder inequality to get
	\[ \norm[p]{\abs{u_n^{p-1} - u^{p-1}}^\frac{1}{p-1}}^{p-1} \leq \norm[p]{u_n - u} \norm[p]{u_n + u}^{p-2}; \]
	if $1 < p < 2$, let $\gamma = p-1 \in (0,1)$ and use \eqref{gamma-inequality} to get
	\[ \norm[p]{\abs{u_n^{p-1} - u^{p-1}}^\frac{1}{p-1}}^{p-1} \leq \norm[p]{u_n - u}^{p-1}. \]
	
	This completes the proof of the first part. Suppose on the contrary that $u \in W^{1,1}(\R^m)$, $u \geq 0$, and that the condition \eqref{finite-integral} holds. Fix $\phi \in C^\infty_c(\R^m)$ and $\eps > 0$. We want to prove that
	\begin{equation} \label{part-eps}
		\int (u + \eps)^\frac{1}{p} \nabla \phi = -\frac{1}{p} \int \phi (u + \eps)^\frac{1-p}{p} \nabla u.
	\end{equation}

	To this end, let $u_n \to u$ in $W^{1,1}(\R^m)$, where $u_n \in C^\infty$, $u_n \geq 0$; up to a subsequence we may suppose also $u_n \to u$ and $\nabla u_n \to \nabla u$ pointwise almost everywhere. Putting $u_n$ in place of $u$ in \eqref{part-eps} we have pointwise convergence of both the integrands, and we conclude via \autoref{generalized-lebesgue} using the dominations
	\[ \abs{\phi (u_n + \eps)^\frac{1-p}{p} \nabla u_n} \leq \eps^\frac{1-p}{p} \abs{\phi} \abs{\nabla u_n}, \quad \abs{\phi (u + \eps)^\frac{1-p}{p} \nabla u} \leq \eps^\frac{1-p}{p} \abs{\phi} \abs{\nabla u}. \]
	
	Finally, letting $\eps \to 0$ in \eqref{part-eps}, we have once again pointwise convergence of the integrands, and we conclude by the classical Lebesgue's dominated covergence Theorem thanks to the hypothesis and the domination
	\[ \abs{\phi (u + \eps)^\frac{1-p}{p} \nabla u}^p \leq \abs{\phi}^p u^{1-p} \abs{\nabla u}^p. \qedhere \]
\end{proof}

Note that the condition \eqref{finite-integral} in \autoref{p-root} is necessary, as the following example shows.

\begin{example} \label{counterexample} In dimension $m=1$, fix $p>1$ and consider the $W^{1,1}$ function
	\[ f(x) = \begin{cases} \sin (x)^{p-1} & 0 \leq x \leq \pi \\ 0 & \text{otherwise,} \end{cases} \]
	whose weak derivative is $f'(x) = \chi_{[0,\pi]} \sin(x)^{p-2} \cos (x)$. The point is that $f^\frac{1}{p}$ does not belong to $W^{1,p}(\R)$, since the weak derivative of $f^\frac{1}{p}$ should be $g_p(x) = \frac{p-1}{p} \chi_{[0,\pi]} \sin(x)^{-\frac{1}{p}} \cos(x)$, but
	\[ \int_0^\pi \abs{g_p(x)^p} \de x = \frac{(p-1)^p}{p^p} \int_0^\pi \frac{\abs{\cos(x)}^p}{\sin(x)} \de x \]
	diverges at both 0 and $\pi$
\end{example}

\begin{prop} \label{p-continuous-maps}
	If $u_n \to u$ in $W^{1,p}(\R^m)$, $u_n, u \geq 0$, then $u_n^p \to u^p$ in $W^{1,1}(\R^m)$.
	
	Viceversa, let $u_n \to u$ in $W^{1,1}(\R^d)$, $u_n, u \geq 0$. Let $h_n,h \in L^1(\R^m)$ such that $u_n^{1-p} \abs{\nabla u_n}^p \leq h_n$, $u^{1-p} \abs{\nabla u}^p \leq h$, and
	\begin{equation} \label{convergence-hypothesis}
		\lim_{n \to \infty} \int h_n = \int h.
	\end{equation}
	Suppose also that for every subsequence $\gra{h_{n_k}}$ there exists a further subsequence converging to $h$ pointwise a.e. Then $u_n^\frac{1}{p} \to u^\frac{1}{p}$ in $W^{1,p}(\R^m)$.
\end{prop}

\begin{proof} 
	If $p = 1$ there is nothing to prove, so assume $p > 1$, and take $u_n \to u$ in $W^{1,p}(\R^m)$. Using \eqref{p-inequality} and the H\"{o}lder inequality with exponents $p$ and $\frac{p}{p-1}$,
	\[ \int \abs{u_n^p - u^p} \leq \norm[p]{u_n - u} \norm[p]{u_n + u}^{p-1}. \]

	Since $u_n \to u$ in $W^{1,p}(\R^m)$ and hence in particular $u_n$ is bounded in $L^p(\R^m)$, we get that $u_n^p \to u^p$ (strongly) in $L^1(\R^m)$. 
	
	Moreover, $\nabla u_n^p = p u_n^{p-1} \nabla u_n$ and $\nabla u^p = p u^{p-1} \nabla u$ by \autoref{p-root}, hence by the H\"{o}lder inequality
	\begin{ieee*}{rCl}
		\int \abs{\nabla u_n^p - \nabla u^p} & \leq & p \int u_n^{p-1} \abs{\nabla u_n - \nabla u} + p \int \abs{\nabla u} \abs{u_n^{p-1} - u^{p-1}} \\
		& \leq & p \norm[p]{u_n}^{p-1} \norm[p]{\nabla u_n - \nabla u} + p \int \abs{\nabla u} \abs{u_n^{p-1} - u^{p-1}},
	\end{ieee*}
	which converges to zero as in the proof of \autoref{p-root}.

	To prove the converse, suppose by contradiction that there is a subsequence (denoted again $u_n$) such that
	\begin{equation} \label{contradiction}
		\norm[W^{1,p}]{u_n^\frac{1}{p}, u^\frac{1}{p}} \geq \delta > 0.
	\end{equation}
	
	By hypothesis, up to a further subsequence we may assume that $u_{n_k} \to u$, $\nabla u_{n_k} \to \nabla u$ and $h_{n_k} \to h$ pointwise almost everywhere. Then we have by \eqref{gamma-inequality}, with $\gamma = \frac{1}{p}$,
	\[ \int \abs{u_{n_k}^\frac{1}{p} - u^\frac{1}{p}}^p \leq \int \abs{u_{n_k} - u} = \norm[1]{u_{n_k} - u}, \]
	and
	\[ \norm[p]{\nabla u_{n_k}^\frac{1}{p} - \nabla u^\frac{1}{p}} = \frac{1}{p^p} \int \abs{u_{n_k}^\frac{1-p}{p} \nabla u_{n_k} - u^\frac{1-p}{p} \nabla u}^p. \]
	Here the integrand converges to zero pointwise, and using the domination
	\[ \abs{u_{n_k}^\frac{1-p}{p} \nabla u_{n_k} - u^\frac{1-p}{p} \nabla u}^p \leq 2^{p-1} \left( u_{n_k}^{1-p} \abs{\nabla u_{n_k}}^p + u^{1-p} \abs{\nabla u}^p \right) \leq 2^{p-1} (h_{n_k} + h) \]
	and the condition \eqref{convergence-hypothesis} we conclude thanks to \autoref{generalized-lebesgue} that $u_{n_k}^\frac{1}{p} \to u^\frac{1}{p}$ in $W^{1,p}(\R^m)$, contradicting \eqref{contradiction}.
\end{proof}


\section{Regular measures} \label{section-regular-measures}

In this Section we study the space $\prob^{1,p}(\R^m)$ of $W^{1,p}$-regular measures. By \autoref{p-root}, it is immediate to see that
\[ \mu \in \prob^{1,p}(\R^m) \implies \mu \in \prob^{1,1}(\R^m), \]
but the converse is not true in general if $p > 1$ (see \autoref{counterexample}). Thus, when $p >1$ we have a strict inclusion $\prob^{1,p}(\R^m) \subsetneq \prob^{1,1}(\R^m)$, .

The set $\prob^{1,p}(\R^m)$ has a natural structure of metric space if endowed with the distance
\[
d^{1,p}(\mu,\nu) = \norm[W^{1,p}]{ \left( \frac{\de \mu}{\de \Leb^m} \right)^{\frac{1}{p}} - \left( \frac{\de \nu}{\de \Leb^m} \right)^{\frac{1}{p}}},
\]
which can be seen as a refined version of the Hellinger distance between two absolutely continuous probability measures, where the $L^p$ norm of the $p$-th roots is replaced by the $W^{1,p}$ norm.

We aim to study the space $\left( \prob^{1,p}\left( \Rdn \right), d^{1,p} \right)$ in relation with the map which sends a $W^{1,p}$-regular probability onto its marginals, namely
\begin{ieee}{rCrCl} 
	\pi & \colon & \prob^{1,p}\left( \Rdn \right) & \longrightarrow & \prob(\R^d)^N \label{pi-definition} \\
	&& \mu & \longmapsto & \left( \proj{\mu}{1}, \dotsc, \proj{\mu}{N} \right). \nonumber
\end{ieee}

In particular we want to prove the two following facts:
\begin{itemize}
	\item if $\mu$ is $W^{1,p}$-regular, then $\proj{\mu}{j}$ is $W^{1,p}$-regular for every $j = 1, \dotsc, N$;
	\item the map $\pi \colon \prob^{1,p}\left( \Rdn \right) \longrightarrow \prob^{1,p}(\R^d)^N$ is continuous with respect to the distance $d^{1,p}$ and the relative product topology on the codomain.
\end{itemize}

These properties will be proved in \autoref{regular-marginals} and \autoref{pi-continuous} respectively. We remark that the latter was alredy proved by Brezis in \cite[Appendix]{lieb2002density} in the case $p=2$. We start by introducing some technical results about the projection map. In what follows, if $\mu$ is $W^{1,p}$-regular, with a slight abuse of notation we will denote by $\mu(X)$ its density, whose $p$-th root belongs to $W^{1,p}\left( \Rdn \right)$. For $j = 1, \dotsc, N$ let 
\begin{equation} \label{rho-definition}
	\proj{\mu}{j}(x_j) = \int \mu(X) \dX[j], \quad \nabla \proj{\mu}{j}(x_j) = \int \nabla_{x_j} \mu(X) \dX[j],
\end{equation}
where $\nabla_{x_j} \mu$ is defined according to \autoref{p-root}. It is easy to prove, approximating $\mu$ with smooth functions in $W^{1,1}\left( \Rdn \right)$, that $\nabla \proj{\mu}{j}$ is the distributional gradient of $\proj{\mu}{j}$, hence $\proj{\mu}{j} \in W^{1,1}(\R^d)$.

\begin{remark}
	Notice that $\proj{\mu}{j}$ coincides with the (density of the) push-forward measure under the projection $\pi^j \colon \Rdn \to \R^d$ on the $j$-th factor, which makes the notation is consistent.
\end{remark}

By \autoref{p-root}, in order to prove that the marginals of a $W^{1,p}$-regular measure are $W^{1,p}$-regular, it suffices to show that
\[ \int \proj{\mu}{j}(x)^{1-p} \abs{\nabla \proj{\mu}{j}(x)}^p \de x \]
is finite.

\begin{lemma} \label{p-domination} Let $\mu \in \prob^{1,p}\left( \Rdn \right)$. Then, for every $j=1, \dotsc, N$,
	\[ \proj{\mu}{j}(x_j)^{1-p} \abs{\nabla \proj{\mu}{j}(x_j)}^p \leq p^p \int \abs{\nabla_{x_j} \mu^\frac{1}{p}(X)}^p \dX[j]. \]
\end{lemma}

\begin{proof}
	Recalling \autoref{p-root} and using the H\"{o}lder inequality with exponents $\frac{p}{p-1}$ and $p$, we get
	\begin{ieee*}{rCl}
		\abs{\nabla \proj{\mu}{j} (x_j)} & \leq & p \int \mu(X)^\frac{p-1}{p} \abs{\nabla_{x_j} \mu^\frac{1}{p}(X)} \dX[j] \\
		& \leq & p \left( \int \mu(X) \dX[j] \right)^{\frac{p-1}{p}} \left( \int \abs{\nabla_{x_j} \mu^\frac{1}{p}(X)}^p \dX[j] \right)^\frac{1}{p} \\
		& = & p \proj{\mu}{j}(x_j)^{\frac{p-1}{p}}\left( \int \abs{\nabla_{x_j} \mu^\frac{1}{p}(X)}^p \dX[j] \right)^\frac{1}{p},
	\end{ieee*}
	which implies the thesis.
\end{proof}

As a corollary we obtain

\begin{thm} \label{regular-marginals}
	Let $\mu \in \prob^{1,p} \left( \Rdn \right)$. Then its marginals belong to $\prob^{1,p}(\R^d)$, and
	\[ \nabla (\proj{\mu}{j})^{\frac{1}{p}} = \frac{1}{p}(\proj{\mu}{j})^{\frac{1-p}{p}} \nabla \proj{\mu}{j}. \]
\end{thm}

\begin{proof} Apply the result of \autoref{p-root} to $\proj{\mu}{j} \in W^{1,1}(\R^d)$, using the domination given by \autoref{p-domination}.
\end{proof}

Finally we want to prove that the map $\pi$ defined in \eqref{pi-definition} is continuous.

\begin{lemma} \label{L1-convergence-marginals}
	Let $\mu^n \to \mu$ in $\prob^{1,p}\left(\Rdn \right)$. Then $\proj{\mu^n}{j} \to \proj{\mu}{j}$ in $L^1(\R^d)$ and $\nabla \proj{\mu^n}{j} \to \nabla \proj{\mu}{j}$ in $L^1(\R^d)^d$.
\end{lemma}

\begin{proof} Using \eqref{p-inequality} and the H\"{o}lder inequality,
	\begin{ieee*}{rCl}
		\int \abs{\proj{\mu^n}{j}(x_j) -\proj{\mu}{j}(x_j)} \de x_j & = & \int \abs{\int \mu(X) - \mu^n(X) \dX[j]} \de x_j \\
		& \leq & \int \abs{\mu^n(X) - \mu(X)} \dX \\
	\end{ieee*}
	and
	\begin{ieee*}{rCl}
		\int \abs{\nabla \proj{\mu^n}{j}(x_j) - \nabla \proj{\mu}{j}(x_j)} \de x_j & = & \int \abs{\int \nabla_{x_j} \mu^n (X) - \nabla_{x_j} \mu (X) \dX[j]} \de x_j \\
		& \leq & \int \abs{\nabla_{x_j} \mu^n (X) - \nabla_{x_j} \mu (X)} \dX.
	\end{ieee*}
	
	We conclude thanks to \autoref{p-root}.
\end{proof}

\begin{thm} \label{pi-continuous} The map $\pi$ is continuous from $\prob^{1,p}\left( \Rdn \right)$ to $\prob^{1,p}(\R^d)^N$ with the product topology.
\end{thm}

\begin{proof}
	Let $\mu^n \to \mu$ in $\prob^{1,p}\left( \Rdn \right)$, and fix $j \in \gra{1, \dotsc, N}$. In order to prove that $\proj{\mu^n}{j} \to \proj{\mu}{j}$ in $\prob^{1,p}(\R^d)$ we want to apply \autoref{p-continuous-maps}, with
	\[ h_n(x_j) = p^p \int \abs{\nabla_{x_j} (\mu^n)^\frac{1}{p}(X)}^p \dX[j], \quad h(x_j) = p^p \int \abs{\nabla_{x_j} \mu^\frac{1}{p}(X)}^p \dX[j]. \]
	
	By \autoref{p-domination} we have $(\proj{\mu^n}{j})^{1-p} \abs{\nabla (\proj{\mu^n}{j})}^p \leq h_n$ and $(\proj{\mu}{j})^{1-p} \abs{\nabla \proj{\mu}{j}}^p \leq h$. Condition \eqref{convergence-hypothesis} is ensured by
	\begin{ieee*}{rCl}
		\lim_{n \to \infty} \int h_n(x_j) \de x_j & = & p^p \lim_{n \to \infty} \int \abs{\nabla_{x_j} (\mu^n)^\frac{1}{p}(X)}^p \dX \\
		& = & p^p \lim_{n \to \infty} \norm[p]{\nabla_{x_j} (\mu^n)^\frac{1}{p}}^p \\
		& = & p^p \norm[p]{\nabla_{x_j} \mu^\frac{1}{p}}^p \\
		& = & \int h(x_j) \de x_j.
	\end{ieee*}
	
	We now follow a construction similar to the one of the Riesz-Fischer theorem, and already used for the analogous result by Brezis in \cite[Appendix]{lieb2002density}. Recall that, by \autoref{p-continuous-maps}, $\mu^n \to \mu$ in $W^{1,1}\left(\Rdn \right)$. For every subsequence (denoted again $h_n$), extract a further subsequence $(h_{n_k})_k$ such that:
	\begin{enumerate}[(i)]
		\item $\nabla (\mu^{n_k})^\frac{1}{p} \to \nabla \mu^\frac{1}{p}$ pointwise a.e.;
		\item $\norm[L^p]{\nabla (\mu^{n_k})^\frac{1}{p} - \nabla \mu^\frac{1}{p}}^p \leq {2^{-k}}$.
	\end{enumerate}
	
	Let
	\[ F(X) = \abs{\nabla \mu^\frac{1}{p}(X)}^p + \sum_{k = 1}^\infty \abs{\nabla (\mu^{n_k})^\frac{1}{p}(X) - \nabla \mu^\frac{1}{p}(X)}^p. \]
	
	Since $F \in L^1\left( \Rdn \right)$ and clearly
  \[ \abs{\nabla (\mu^{n_k})^\frac{1}{p}(X)}^p \leq 2^{p-1}F(X), \quad \abs{\nabla \mu^\frac{1}{p}(X)}^p \leq F(X) \]
  we have that $h_{n_k} \to h$ pointwise a.e. by dominated convergence. Finally $\proj{\mu^n}{j} \to \proj{\mu}{j}$ in $W^{1,1}(\R^d)$ by \autoref{L1-convergence-marginals}, and we may conclude by \autoref{p-continuous-maps}.
\end{proof}

\section{Energy of regular measures}

If $\mu \in \prob^{1,p}(\R^m)$, it will be useful to deal with the Sobolev norm of $\mu^\frac{1}{p}$. However, since $\mu$ is a probability,
\[ \norm[W^{1,p}]{\mu^\frac{1}{p}}^p = \int \mu(x) \de x + \int \abs{\nabla \mu^\frac{1}{p}(x)}^p \de x = 1 + \int \abs{\nabla \mu^\frac{1}{p}(x)}^p \de x, \]
so all the information is contained in the second summand. Therefore we give the following

\begin{mydef} \label{energy-definition}
	If $\mu \in \prob^{1,p}$, the \emph{$W^{1,p}$-energy} of $\mu$ is defined as
	\begin{equation} 
		\energy{1,p}(\mu) = \int \abs{\nabla \mu^\frac{1}{p} (x)}^p \de x.
	\end{equation}
\end{mydef}

In the special case $p=2$, this quantity may be seen as the kinetic energy $\int \abs{\nabla \psi}^2$ of a system described by a wave-function $\psi \in W^{1,2}(\R^m)$, which justifies the name. It is well-known (see for instance \cite{lieb2002density}) that the kinetic energy of a wave-function is bounded from below by (a constant times) the kinetic energy of its marginals. This is also true in our setting, as stated in the following

\begin{lemma} \label{energy-lemma} Let $\mu \in \prob^{1,p}\left( \Rdn \right)$. Then
	\[ \energy{1,p}(\mu) \geq \sum_{j = 1}^N \energy{1,p}(\proj{\mu}{j}). \]
	
	Moreover, if $\rho_1, \dotsc, \rho_N \in \prob^{1,p}(\R^d)$,
	\[ \inf \gra{\energy{1,p}(\mu) \st \mu \in \prob^{1,p}(\R^m) \cap \Pi(\rho_1, \dotsc, \rho_N)} = \sum_{j = 1}^N \energy{1,p}(\rho_j).  \]
\end{lemma}

\begin{proof}
	Fix $\mu \in \prob^{1,p}\left( \Rdn \right)$. By \autoref{regular-marginals} and \autoref{p-domination} we have
	\[ \abs{\nabla (\proj{\mu}{j})^\frac{1}{p} (x_j)}^p = \frac{1}{p^p} \proj{\mu}{j}(x_j)^{1-p} \abs{\nabla \proj{\mu}{j}(x_j)}^p \leq \int \abs{\nabla_{x_j} \mu^\frac{1}{p} (X)}^p \dX[j]. \]
	
	Summing on $j$ and recalling the condition \eqref{abs-gradient-convention} we get the thesis. As for the second statement, due to the first one clearly we have
	\[ \inf \gra{\energy{1,p}(\mu) \st \mu \in \Pi(\rho_1, \dotsc, \rho_N)} \geq \sum_{j = 1}^N \energy{1,p}(\rho_j).  \]
	
	Let however $\mu(X) \eqdef \rho_1(x_1) \dotsm \rho_N(x_N)$; then $\mu$ is such that $\mu \in \prob^{1,p}\left( \Rdn \right)$ and
	\[ \nabla_{x_j} \mu^\frac{1}{p} = \nabla \rho_j^\frac{1}{p} \prod_{\substack{k = 1 \\ k \neq j}}^N \rho_k(x_k)^\frac{1}{p}; \]
	hence
	\[ \int \abs{\nabla_{x_j} \mu^\frac{1}{p}(X)}^p \dX = \int \abs{\nabla \rho_j^\frac{1}{p} (x_j)}^p \de x_j = \energy{1,p}(\rho_j). \]
	
	Finally summing on $j$ and taking into account the usual condition \eqref{abs-gradient-convention},
	\[ \energy{1,p}(\mu) = \sum_{j = 1}^N \energy{1,p}(\rho_j). \qedhere \]
\end{proof}

\begin{prop} \label{energy-convergence}
	Let $\eta \in C^\infty(\R^m)$, $\eta \geq 0$ such that $\int \eta = 1$ and define for $\eps > 0$
	\[ \eta^\eps(x) = \frac{1}{\eps^m} \eta\left( \frac x\eps \right). \]
	
	Then, for every $\mu \in \prob^{1,p}(\R^m)$,
	\[ \energy{1,p}(\mu * \eta^\eps) \leq \energy{1,p}(\mu) \quad \text{and} \quad \lim_{\eps \to 0} \energy{1,p}(\mu * \eta^\eps) = \energy{1,p}(\mu). \]
\end{prop}

\begin{proof} By the H\"{o}lder inequality with exponents $p$ and $\frac{p}{p-1}$ we have
	\begin{ieee*}{rCl}
		\abs{\nabla (\mu * \eta^\eps)(x)} & = & \abs{((\nabla \mu) * \eta^\eps)(x)} \\
		& \leq & \int \abs{\nabla \mu(y)} \eta^\eps(x-y) \de y \\
		& \leq & \left( \int \mu(y)^{1-p} \abs{\nabla \mu(y)}^p \eta^\eps(x-y) \de y \right)^{\frac{1}{p}} (\mu * \eta^\eps)(x)^\frac{p-1}{p}.
	\end{ieee*}
	
	Since $\mu * \eta^\eps \in C^\infty(\R^m)$ we have
	\begin{ieee}{rCl}
		\abs{\nabla (\mu * \eta^\eps)(x)^\frac{1}{p}} & = & \frac{1}{p} (\mu * \eta^\eps)(x)^\frac{1-p}{p} \abs{\nabla (\mu * \eta^\eps)(x)} \nonumber \\
		& \leq & \frac{1}{p} \left( \int \mu(y)^{1-p} \abs{\nabla \mu(y)}^p \eta^\eps(x-y) \de y \right)^{\frac{1}{p}}, \label{eta-eps-domination}
	\end{ieee} 
	whence
	\begin{ieee*}{rCl}
		\energy{1,p}(\mu * \eta^\eps) & = & \int \abs{\nabla (\mu * \eta^\eps)^\frac{1}{p}(x)}^p \de x \\
		& \leq & \frac{1}{p^p} \int \mu(y)^{1-p} \abs{\nabla \mu(y)}^p \eta^\eps(x-y) \de y \de x \\
		& = & \int \abs{\nabla \mu^\frac{1}{p} (y)}^p \de y = \energy{1,p}(\mu).
	\end{ieee*}

	In order to prove the second part, it suffices to show that $(\mu * \eta^\eps)^\frac{1}{p}$ converges strongly to $\mu^\frac{1}{p}$ in $W^{1,p}$ to get that
	\[ \lim_{\eps \to 0} \energy{1,p}(\mu * \eta^\eps) = \lim_{\eps \to 0} \norm[W^{1,p}]{(\mu * \eta^\eps)^\frac{1}{p}}^p - 1 = \norm[W^{1,p}]{\mu^\frac{1}{p}} - 1 = \energy{1,p}(\mu). \]
	
	Since $(\mu^{1-p} \abs{\nabla \mu}^p) * \eta^\eps \longrightarrow \mu^{1-p} \abs{\nabla \mu}^p$ pointwise a.e., inequality \eqref{eta-eps-domination} gives a domination which allows to conclude thanks to \autoref{p-continuous-maps}.
\end{proof}


\section{Definition of the smoothing operator} \label{section-definition}

In this section we start to deal with the main problem of the paper, which we recall here.

\medskip
\textbf{Problem: } \textit{Given $\rho_1, \dotsc, \rho_N\in \prob^{1,p}(\R^d)$, and given $\mu \in \Pi(\rho_1, \dotsc, \rho_N)$, find a family $\left( \mu^\eps \right)_{\eps > 0}$ such that:
	\begin{enumerate}[(i)]
		\item $\mu^\eps \in \Pi(\rho_1, \dotsc, \rho_N);$
		\item $\mu^\eps \in \prob^{1,p} \left( \Rdn \right);$
		\item $\mu^\eps \wconv \mu$ as $\eps \to 0$.
\end{enumerate}}
\medskip

To this end, we will define an operator
\begin{ieee*}{rCcCc}
	\Theta & \colon & \R^+ \times \prob\left( \Rdn \right) & \longrightarrow & \prob\left( \Rdn \right) \\
	&& (\eps, \mu) & \longmapsto & \Theta^\eps[\mu]
\end{ieee*}
such that:
\begin{enumerate}[A.]
	\item \label{Theta-marginals} for every $\eps > 0$, for every $j = 1, \dotsc, N$,
	\[ \proj{\Theta^\eps[\mu]}{j} = \proj{\mu}{j}; \]
	\item \label{Theta-regularity} if $\proj{\mu}{j} \in \prob^{1,p}(\R^d)$ for evey $j = 1, \dotsc, N$, then
	\[ \Theta^\eps[\mu] \in \prob^{1,p}\left( \Rdn \right); \] 
	\item \label{Theta-Cb-continuity} for every $\phi \in C_b\left(\Rdn\right)$,
	\[ \lim_{\eps \to 0} \int \phi \de \Theta^\eps[\mu] = \int \phi \de \mu. \]
\end{enumerate}

This will give a universal construction which solves the problem: properties \ref{Theta-marginals}--\ref{Theta-Cb-continuity} ensure that, taking $\mu^\eps \eqdef \Theta^\eps[\mu]$, the requirements (i)--(iii) above are satisfied. Moreover, the smoothing operator $\Theta$ will also satisfy the following form of continuity with respect to the measure argument.

\begin{restatable}{thm}{mucontinuity} \label{mu-continuity}
	Let $\mu^n, \mu \in \prob\left( \Rdn \right)$ such that:
	\begin{enumerate}[(i)]
		\item $\mu^n \rightharpoonup \mu$ in duality with $C_b\left( \Rdn \right);$
		\item for every $j = 1, \dotsc, N$, $\proj{\mu^n}{j} \in \prob^{1,p}(\R^d)$ and $\proj{\mu}{j} \in \prob^{1,p}(\R^d)$, with
		\[ \lim_{n \to \infty} d^{1,p} \left( \proj{\mu^n}{j}, \proj{\mu}{j} \right) = 0. \]
	\end{enumerate} 
	
	Then, for every $\eps > 0$, 
	\[ \lim_{n \to \infty} d^{1,p} \left( \Theta^\eps[\mu^n] , \Theta^\eps[\mu] \right) = 0. \]
\end{restatable}

The proof of \autoref{mu-continuity} will be presented in \autoref{section-mu-continuity}.

\bigskip

Now we proceed with the construction of the smoothing operator $\Theta$. Given $\eps > 0$, let $\eta^\eps \colon \R^d \to \R^+$ be
\[ \eta^\eps(z) = \frac{1}{(2\pi\eps)^{d/2}} \exp \left( -\frac{\abs{z}^2}{2\eps} \right). \]

For $\mu \in \prob\left( \Rdn \right)$, we define the measure $\Lambda^\eps[\mu]$ as the convolution of $\mu$ with the kernel $\eta^\eps(x_1) \dotsm \eta^\eps(x_N)$, \ie, if $\psi \colon \Rdn \to \R$ is any continuous bounded function,
\begin{equation} \label{Lambda-definition}
	\int \psi(Y) \de \Lambda^\eps[\mu](Y) \eqdef \iint \psi(Y) \prod_{k = 1}^N \eta^\eps(y_k-x_k) \de \mu(X) \dY.
\end{equation}

Notice that $\Lambda^\eps[\mu]$ is absolutely continuous with respect to the Lebesgue measure, with density
\[ \Lambda^\eps[\mu](Y) = \int \prod_{k = 1}^N \eta^\eps(y_k-x_k) \de \mu(X). \] 

Finally, if $\psi \colon \Rdn \to \R$ is any continuous bounded function, we define $\Theta^\eps[\mu]$ via the expression
\begin{equation} \label{Theta-definition}
	\int \psi(X) \de \Theta^\eps[\mu](X) \eqdef \iint \psi(X) \prod_{k = 1}^N \frac{\eta^\eps(y_k-x_k)}{(\proj{\mu}{k} * \eta^\eps)(y_k)} \de \proj{\mu}{k} (x_k) \de \Lambda^\eps[\mu](Y).
\end{equation}

Here, with a slight abuse of notation, the denominator $(\proj{\mu}{k} * \eta^\eps)(y_k)$ denotes the density of the measure $\proj{\mu}{k} * \eta^\eps$ evaluated at $y_k$, namely
\[ (\proj{\mu}{k} * \eta^\eps)(y_k) = \int \eta^\eps(y_k-x_k) \de \proj{\mu}{k}(x_k), \]
and is always strictly positive, since $\proj{\mu}{k}$ is a probability and $\eta^\eps > 0$.

\begin{remark} This construction fits into the general framework for the composition of transport plans, as in \cite[Section 5.3]{ambrosio2008gradient}. Indeed, the definition of $\Theta^\eps[\mu]$ may be seen as follows: as a first step we regularize $\mu$ by convolution; secondly, we consider the 2-transport plans $\beta_j$ for $j = 1, \dotsc, N$ defined by
\[ \int \phi(x,y) \de \beta_j(x,y) = \int \phi(x,y) \eta^\eps(x-y) \de \proj{\mu}{j}(y) \de y \]
for any $\phi \in C_b(\R^d \times \R^d)$. Notice that $\beta_j$ has marginals $\proj{\mu}{j} * \eta^\eps$ and $\proj{\mu}{j}$. Then $\Theta^\eps[\mu]$ corresponds to the composition of $\Lambda^\eps[\mu]$ with $\beta_j$ on each corresponding $j$-th marginal.
\end{remark}

\begin{lemma}[Property \ref{Theta-marginals}] \label{basic-properties} Let $\mu \in \prob\left( \Rdn \right)$. Then for every $\eps > 0$ and for every $j = 1, \dotsc, N$ the following hold.
\begin{enumerate}[(i)]
 \item $\proj{\Lambda^\eps[\mu]}{j} = \proj{\mu}{j} * \eta^\eps;$
 \item $\proj{\Theta^\eps[\mu]}{j} = \proj{\mu}{j}$.
\end{enumerate}
\end{lemma}

\begin{proof} (i) If $\phi \in C_b(\R^d)$, by the Fubini's Theorem we have
\begin{align*}
 \int \phi(y_j) \de \proj{\Lambda^\eps[\mu]}{j}(y_j) &= \int \phi(y_j) \de \Lambda^\eps[\mu](Y) \\
  &= \int \phi(y_j) \prod_{k = 1}^N \eta^\eps(y_k-x_k) \de \mu(X) \dY \\
  &= \int \phi(y_j) \eta^\eps(y_j-x_j) \de \mu(X) \de y_j \\
  &= \int \phi(y_j) \eta^\eps(y_j-x_j) \de \proj{\mu}{j}(x_j) = \int \phi(y_j) \de (\proj{\mu}{j} * \eta^\eps) (y_j).
\end{align*}

(ii) Using (i), if $\phi \in C_b(\R^d)$, by the Fubini's Theorem we have
\begin{align*}
 \int \phi(x_j) \de (\proj{\Theta^\eps[\mu]}{j})(X) &= \int \phi(x_j) \de (\Theta^\eps[\mu])(X) \\
 &= \iint \phi(x_j) \prod_{k = 1}^N \frac{\eta^\eps(y_k-x_k)}{(\proj{\mu}{k} * \eta^\eps) (y_k)} \de \proj{\mu}{k} (x_k) \de \Lambda^\eps[\mu](Y) \\
 &= \int \phi(x_j) \frac{\eta^\eps(y_j-x_j)}{(\proj{\mu}{j} * \eta^\eps) (y_j)} \de \Lambda^\eps[\mu](Y) \de \proj{\mu}{j} (x_j) \\
 &= \int \phi(x_j) \frac{\eta^\eps(y_j-x_j)}{(\proj{\mu}{j} * \eta^\eps) (y_j)} \de (\proj{\mu}{j} * \eta^\eps) (y_j) \de \proj{\mu}{j} (x_j) \\
 &= \int \phi(x_j) \de \proj{\mu}{j} (x_j). \tag*{\qedhere}
\end{align*}
\end{proof}


\section{Regularity of $\Theta$} \label{section-regularity}

In this Section we prove that $\Theta$ satisfies property \ref{Theta-regularity} of Section \ref{section-definition}. Moreover, some additional estimates on the $W^{1,p}$-energy of $\Theta^\eps[\mu]$ also hold. Let $\mu \in \prob\left( \Rdn \right)$ such that $\proj{\mu}{j} \in \prob^{1,p}(\R^d)$ for every $j = 1, \dotsc, N$. Then $\Theta^\eps[\mu]$ is absolutely continuous with respect to the Lebesgue measure, with density given by
\[ \Theta^\eps[\mu](X) = \int P^\eps[\mu](X,Y) \dY, \]
where we denote by $P^\eps[\mu]$ the integral kernel appearing in \eqref{Theta-definition}, namely
\begin{equation}
	P^\eps[\mu](X, Y) \eqdef \prod_{k = 1}^N \frac{\eta^\eps(y_k-x_k)}{(\proj{\mu}{k} * \eta^\eps) (y_k)} \proj{\mu}{k}(x_k) \Lambda^\eps[\mu](Y).
\end{equation}

Let us denote by
\begin{ieee}{rCl}
	\nabla_{x_j} \Theta^\eps[\mu](X) & \eqdef & \frac{\nabla \proj{\mu}{j}(x_j)}{\proj{\mu}{j}(x_j)} \Theta^\eps[\mu](X) \nonumber \\
	&& {} - \int \frac{\nabla \eta^\eps(y_j - x_j)}{\eta^\eps(y_j-x_j)} P^\eps[\mu](X, Y) \dY. \label{nabla-j-definition}
\end{ieee}

We claim that $\nabla_{x_j} \Theta^\eps[\mu](X)$ is the weak gradient with respect to the $j$-th variable of $\Theta^\eps[\mu](X)$ in $W^{1,1}(\Rdn)$. Indeed, if $\psi \in C^\infty_c(\Rdn)$, by the Fubini's Theorem we may perform first the integration in $x_j$ to get
\begin{ieee*}{rCl}
	\int \nabla_{x_j} \psi(X) \Theta^\eps[\mu](X) \dX & = & \iint \nabla_{x_j} \psi(X) P^\eps[\mu](X, Y) \dX \dY \\
	& = & \iint \psi(X) \frac{\nabla \proj{\mu}{j}(x_j)}{\proj{\mu}{j}(x_j)} P^\eps[\mu](X, Y) \dX \dY \\
	&& {} - \iint \psi(X) \frac{\nabla \eta^\eps(y_j - x_j)}{\eta^\eps(y_j-x_j)} P^\eps[\mu](X, Y) \dX \dY \\
	& = & \int \psi(X) \frac{\nabla \proj{\mu}{j}(x_j)}{\proj{\mu}{j}(x_j)} \Theta^\eps[\mu](X) \dX \\
	&& {} - \int \psi(X) \int \frac{\nabla \eta^\eps(y_j - x_j)}{\eta^\eps(y_j-x_j)} P^\eps[\mu](X, Y) \dY \dX.
\end{ieee*}

To conclude that $\Theta^\eps[\mu] \in \prob^{1,p}\left(\Rdn\right)$, in view of \autoref{p-root}, it suffices to show a suitable domination, which is given by the following

\begin{lemma} \label{nabla-domination} Let $\mu \in \prob(\Rdn)$ such that $\proj{\mu}{j} \in \prob^{1,p}(\R^d)$ for every $j = 1, \dotsc, N$. Then
	\begin{ieee*}{l}
		\abs{\nabla_{x_j} \Theta^\eps[\mu](X)}^p \Theta^\eps[\mu](X)^{1-p} \\
		\leq 2^{p-1} \left(\frac{\abs{\nabla \proj{\mu}{j}(x_j)}^p}{\proj{\mu}{j}(x_j)^p} \Theta^\eps[\mu](X) + \int \frac{\abs{\nabla \eta^\eps(y_j - x_j)}^p}{\eta^\eps(y_j-x_j)^p} P^\eps[\mu](X, Y) \dY \right)
	\end{ieee*}
\end{lemma}

\begin{proof} By the triangular inequality we immediately get 
	\[ \abs{\nabla_{x_j} \Theta^\eps[\mu](X)} \leq \frac{\abs{\nabla \proj{\mu}{j}(x_j)}}{\proj{\mu}{j}(x_j)} \Theta^\eps[\mu](X) + \int \frac{\abs{\nabla \eta^\eps(y_j - x_j)}}{\eta^\eps(y_j-x_j)} P^\eps[\mu](X, Y) \dY. \]
	
	Using the H\"{o}lder inequality with exponents $p$ and $\frac{p}{p-1}$,
	\begin{ieee*}{l}
		\int \frac{\abs{\nabla \eta^\eps(y_j - x_j)}}{\eta^\eps(y_j-x_j)} P^\eps[\mu](X, Y) \dY \\
		\leq \left( \int \frac{\abs{\nabla \eta^\eps(y_j - x_j)}^p}{\eta^\eps(y_j-x_j)^p} P^\eps[\mu](X, Y) \dY \right)^\frac{1}{p} \Theta^\eps[\mu](X)^\frac{p-1}{p},
	\end{ieee*}	
	and the thesis follows.
\end{proof}

Finally we get the proof of property \ref{Theta-regularity}, together with the usual explicit formula for the weak gradient of $\Theta^\eps[\mu]^\frac{1}{p}$.

\begin{thm}[Property \ref{Theta-regularity}] \label{thm-regularity}
	Let $\mu \in \prob\left( \Rdn \right)$ such that $\proj{\mu}{j} \in \prob^{1,p}(\R^d)$ for every $j = 1, \dotsc, N$. Then $\Theta^\eps[\mu] \in \prob^{1,p}\left( \Rdn \right)$, and
	\[ \nabla_{x_j} \Theta^\eps[\mu]^{\frac{1}{p}} (X) = \frac{1}{p} \Theta^\eps[\mu](X)^{\frac{1-p}{p}} \nabla_{x_j} \Theta^\eps[\mu](X). \]
\end{thm}

\begin{proof}
	Recalling \autoref{p-root}, it suffices to check that condition \eqref{finite-integral} holds. Using \autoref{nabla-domination} we have
	\begin{ieee*}{l}
		\int \abs{\nabla_{x_j} \Theta^\eps[\mu](X)}^p \Theta^\eps[\mu](X)^{1-p} \dX \\
		\leq 2^{p-1} \left(\int \frac{\abs{\nabla \proj{\mu}{j}(x_j)}^p}{\proj{\mu}{j}(x_j)^p} \Theta^\eps[\mu](X) \dX + \iint \frac{\abs{\nabla \eta^\eps(y_j - x_j)}^p}{\eta^\eps(y_j-x_j)^p} P^\eps[\mu](X, Y) \dY \dX \right) \\
		= 2^{p-1} \left(\int \frac{\abs{\nabla \proj{\mu}{j}(x_j)}^p}{\proj{\mu}{j}(x_j)^{p-1}} \de x_j + \int \frac{\abs{\nabla \eta^\eps(z)}^p}{\eta^\eps(z)^{p-1}} \de z \right) \\
		= 2^{p-1} p^p \norm[p]{\nabla (\proj{\mu}{j})^\frac{1}{p}}^p + C(d,\eps,p),
	\end{ieee*}
	where the latter is a constant depending only on the dimension $d$, the exponent $p$ and $\eps$.
\end{proof}

From \autoref{thm-regularity} we get also some estimates on the $W^{1,p}$-energy of $\Theta^\eps[\mu]$. In the case $p=2$ the Hilbertian structure allows to simplify some computation and to get sharper constants.

\begin{thm} \label{energy-estimate}
	Let $\mu \in \prob\left( \Rdn \right)$ such that $\proj{\mu}{j} \in \prob^{1,2}(\R^d)$ for every $j = 1, \dotsc, N$. Then
	\begin{equation} \label{energy-bound-1}
		\energy{1,2}(\Theta^\eps[\mu]) \leq \sum_{j = 1}^N \energy{1,2}(\proj{\mu}{j}) + \frac{Nc(d)}{\eps},
	\end{equation}
	where $c(d)$ is a constant depending only on the dimension $d$.
	
	If in addition $\mu \in \prob^{1,2}\left( \Rdn \right)$, then
	\begin{equation} \label{energy-bound-2}
		\energy{1,2}(\Theta^\eps[\mu]) \leq \sum_{j = 1}^N \left( \norm[2]{\nabla_{x_j} \sqrt{\Lambda^\eps[\mu]}} + \Delta(\eps,\mu) \right)^2.
	\end{equation}
	where
	\[ \Delta(\eps,\mu) = \sqrt{\energy{1,2}(\proj{\mu}{j}) - \energy{1,2}(\proj{\mu}{j} * \eta^\eps)}. \]
\end{thm}

\begin{proof}[Proof of \autoref{energy-estimate}] By \autoref{thm-regularity} we have
	\begin{ieee*}{rCl}
		\energy{1,2}(\Theta^\eps[\mu]) & = & \int \abs{\nabla \sqrt{\Theta^\eps[\mu]} (X)}^2 \dX \\
		& = & \frac{1}{4} \int \frac{\abs{\nabla \Theta^\eps[\mu](X)}^2}{\Theta^\eps[\mu](X)} \dX \\
		& = & \frac{1}{4} \sum_{j = 1}^N \int \frac{\abs{\nabla \proj{\mu}{j}(x_j)}^2}{\proj{\mu}{j}(x_j)^2} \Theta^\eps[\mu](X) \dX \\
		&& {} - \sum_{j = 1}^N \iint \frac{\nabla \proj{\mu}{j}(x_j) \cdot \nabla \eta^\eps(y_j - x_j)}{\proj{\mu}{j}(x_j)\eta^\eps(y_j-x_j)} P^\eps[\mu](X, Y) \dX \dY \\
		&& {} + \sum_{j = 1}^N \int \frac{1}{\Theta^\eps[\mu](X)} \abs{\int \frac{\nabla \eta^\eps(y_j - x_j)}{\eta^\eps(y_j-x_j)} P^\eps[\mu](X, Y) \dY}^2 \dX.
	\end{ieee*}
	
	We treat the three terms in order. First we have
	\begin{ieee*}{rCl}
		\frac{1}{4} \sum_{j = 1}^N \int \frac{\abs{\nabla \proj{\mu}{j}(x_j)}^2}{\proj{\mu}{j}(x_j)^2} \Theta^\eps[\mu](X) \dX & = & \frac{1}{4} \sum_{j = 1}^N \int \frac{\abs{\nabla \proj{\mu}{j}(x_j)}^2}{\proj{\mu}{j}(x_j)} \de x_j \\
		& = & \sum_{j = 1}^N \int \abs{\nabla \sqrt{\proj{\mu}{j}} (x_j)}^2 \de x_j \\
		& = & \sum_{j = 1}^N \energy{1,2}(\proj{\mu}{j}).
	\end{ieee*}

	The middle term vanishes. Indeed, using Fubini's theorem and a change of variables,
	\begin{ieee*}{l}
		\iint \frac{\nabla \proj{\mu}{j}(x_j) \cdot \nabla \eta^\eps(y_j - x_j)}{\proj{\mu}{j}(x_j)\eta^\eps(y_j-x_j)} P^\eps[\mu](X, Y) \dX \dY \\
		= \iint \frac{\nabla \proj{\mu}{j}(x_j) \cdot \nabla \eta^\eps(y_j - x_j)}{\proj{\mu}{j}(x_j)\eta^\eps(y_j-x_j)} \frac{\eta^\eps(y_j-x_j)}{(\proj{\mu}{j} * \eta^\eps)(y_j)} \proj{\mu}{j}(x_j) \Lambda^\eps[\mu](Y) \de x_j \dY \\
		= \iint \nabla \proj{\mu}{j}(x_j) \cdot \nabla \eta^\eps(y_j - x_j) \de x_j \de y_j \\
		= \left( \iint \nabla \proj{\mu}{j}(x_j) \de x_j \right) \cdot \left( \int \nabla \eta^\eps(z) \de z \right), 
	\end{ieee*}
	and the second term is zero, as it can be seen, for instance, integrating in spherical coordinates.
	
	Finally, by the Cauchy-Schwarz inequality,
	\begin{ieee*}{l}
		\abs{\int \frac{\nabla \eta^\eps(y_j - x_j)}{\eta^\eps(y_j-x_j)} P^\eps[\mu](X, Y) \dY}^2 \\
		\leq \left( \int \frac{\abs{\nabla \eta^\eps(y_j - x_j)}^2}{\eta^\eps(y_j-x_j)^2} P^\eps[\mu](X, Y) \dY \right) \int P^\eps[\mu](X, Y) \dY \\
		= \Theta^\eps[\mu](X) \int \frac{\abs{\nabla \eta^\eps(y_j - x_j)}^2}{\eta^\eps(y_j-x_j)^2} P^\eps[\mu](X, Y) \dY.
	\end{ieee*}

	Hence the third term is bounded by
	\begin{ieee*}{l}
		\sum_{j = 1}^N \iint \frac{\abs{\nabla \eta^\eps(y_j - x_j)}^2}{\eta^\eps(y_j-x_j)^2} P^\eps[\mu](X, Y) \dY \dX \\
		= \sum_{j = 1}^N \iint \frac{\abs{\nabla \eta^\eps(y_j - x_j)}^2}{\eta^\eps(y_j-x_j)} \proj{\mu}{j}(x_j) \Lambda^\eps[\mu](Y) \dY \de x_j \\
		= \sum_{j = 1}^N \iint \frac{\abs{\nabla \eta^\eps(y_j - x_j)}^2}{\eta^\eps(y_j-x_j)} \proj{\mu}{j}(x_j) \de y_j \de x_j \\
		= N \int \frac{\abs{\nabla \eta^\eps(z)}^2}{\eta^\eps(z)} \de z = \frac{Nc(d)}{\eps},
	\end{ieee*}
	where $c(d)$ is a constant depending only on the dimension $d$.
	
	In order to show the second part of the statement, notice that, if $\mu$ is $W^{1,p}$-regular, performing a change of variables in \eqref{nabla-j-definition} we may write
	\begin{ieee}{rCl}
		\nabla_{x_j} \Theta^\eps[\mu](X) & = & \int \left( \frac{\nabla \proj{\mu}{j}(x_j)}{\proj{\mu}{j}(x_j)} - \frac{\nabla (\proj{\mu}{j} * \eta^\eps)(y_j)}{(\proj{\mu}{j} * \eta^\eps)(y_j)}\right) P^\eps[\mu](X, Y) \dY \nonumber \\
		&& {} + \int \nabla_{x_j} \Lambda^\eps[\mu](Y) \prod_{j = 1}^N \frac{\eta^\eps(y_j - x_j)}{(\proj{\mu}{j} * \eta^\eps)(y_j)} \proj{\mu}{j}(x_j) \dY \label{gradient-formula} \\
		& = & \vcentcolon I(X) + II(X) \nonumber
	\end{ieee}
	
	We estimate both terms via the Cauchy-Schwarz inequality to get
	\begin{ieee*}{rCl}
		\abs{I(X)}^2 & \leq & \left( \int \abs{\frac{\nabla \proj{\mu}{j}(x_j)}{\proj{\mu}{j}(x_j)} - \frac{\nabla (\proj{\mu}{j} * \eta^\eps)(y_j)}{(\proj{\mu}{j} * \eta^\eps)(y_j)}}^2 P^\eps[\mu](X, Y) \dY \right) \int P^\eps[\mu](X, Y) \dY\\
		& = & \Theta^\eps[\mu](X) \int \abs{\frac{\nabla \proj{\mu}{j}(x_j)}{\proj{\mu}{j}(x_j)} - \frac{\nabla (\proj{\mu}{j} * \eta^\eps)(y_j)}{(\proj{\mu}{j} * \eta^\eps)(y_j)}}^2 P^\eps[\mu](X, Y) \dY,
	\end{ieee*}
	and
	\begin{ieee*}{rCl}
		\abs{II(X)}^2 & \leq & \left( \int \frac{\abs{\nabla_{x_j} \Lambda^\eps[\mu](Y)}^2}{\Lambda^\eps[\mu](Y)^2} P^\eps[\mu](X,Y) \dY \right) \int P^\eps[\mu](X,Y) \dY\\
		& = & \Theta^\eps[\mu](X) \int \frac{\abs{\nabla_{x_j} \Lambda^\eps[\mu](Y)}^2}{\Lambda^\eps[\mu](Y)^2} P^\eps[\mu](X,Y) \dY.
	\end{ieee*}
	
	It follows that
	\begin{ieee*}{rCl}
		\int \frac{\abs{I(X)}^2}{\Theta^\eps[\mu](X)} \dX & \leq & \iint \abs{\frac{\nabla \proj{\mu}{j}(x_j)}{\proj{\mu}{j}(x_j)} - \frac{\nabla (\proj{\mu}{j} * \eta^\eps)(y_j)}{(\proj{\mu}{j} * \eta^\eps)(y_j)}}^2 P^\eps[\mu](X,Y) \dX \dY \\
		& = & \iint \abs{\frac{\nabla \proj{\mu}{j}(x_j)}{\proj{\mu}{j}(x_j)} - \frac{\nabla (\proj{\mu}{j} * \eta^\eps)(y_j)}{(\proj{\mu}{j} * \eta^\eps)(y_j)}}^2 \eta^\eps(y_j-x_j) \proj{\mu}{j}(x_j) \de x_j \de y_j \\
		& = & \int \frac{\abs{\nabla \proj{\mu}{j}(x_j)}^2}{\proj{\mu}{j}(x_j)} \de x_j - \int \frac{\abs{\nabla (\proj{\mu}{j} * \eta^\eps)(y_j)}^2}{(\proj{\mu}{j} * \eta^\eps)(y_j)} \de y_j \\
		& = & 4 \energy{1,2}(\proj{\mu}{j}) - 4 \energy{1,2}(\proj{\mu}{j} * \eta^\eps)
	\end{ieee*}
	and
	\begin{ieee*}{rCl}
		\int \frac{\abs{I(X)}^2}{\Theta^\eps[\mu](X)} \dX & \leq & \iint \frac{\abs{\nabla_{x_j} \Lambda^\eps[\mu](Y)}^2}{\Lambda^\eps[\mu](Y)^2} P^\eps[\mu](X,Y) \dX \dY \\
		& = & \int \frac{\abs{\nabla_{x_j} \Lambda^\eps[\mu](Y)}^2}{\Lambda^\eps[\mu](Y)} \dY
	\end{ieee*}
	
	Hence, for every $\tau_j > 0$,
	\begin{ieee*}{l}
		\int \abs{\nabla_{x_j} \sqrt{\Theta^\eps[\mu]}(X)}^2 \dX \\
		\leq (1 + \tau_j) \frac{1}{4} \int \frac{\abs{I(X)}^2}{\Theta^\eps[\mu](X)} \dX + \left(1 + \tau_j^{-1} \right) \frac{1}{4} \int \frac{\abs{II(X)}^2}{\Theta^\eps[\mu](X)} \dX \\
		\leq (1 + \tau_j) \left( \energy{1,2}(\proj{\mu}{j}) - \energy{1,2}(\proj{\mu}{j} * \eta^\eps) \right) + \left(1 + \tau_j^{-1} \right) \int \frac{\abs{\nabla_{x_j} \Lambda^\eps[\mu](Y)}^2}{4 \Lambda^\eps[\mu](Y)} \dY.
	\end{ieee*} 

	Optimizing in $\tau_j$ and summing over $j = 1, \dotsc, N$ we get the thesis.
\end{proof}

\begin{thm} \label{p-energy-estimate}
	Let $\mu \in \prob\left( \Rdn \right)$. Then there exists a constant $c(d,p)$ depending on the dimension $d$ and the exponent $p$ such that
	\begin{equation} \label{p-energy-bound-1}
		\energy{1,p}(\Theta^\eps[\mu]) \leq \sum_{j = 1}^N \left( \energy{1,p}(\proj{\mu}{j})^\frac{1}{p} + \frac{c(d,p)}{\sqrt{\eps}} \right)^p.
	\end{equation}
	
	If in addition $\mu \in \prob^{1,p}\left( \Rdn \right)$ and $p > 1$, then
	\begin{equation} \label{p-energy-bound-2}
		\energy{1,p}(\Theta^\eps[\mu]) \leq \sum_{j = 1}^N \left( \norm[p]{\nabla_{x_j} \Lambda^\eps[\mu]^\frac{1}{p}} + c_p \Delta(\eps,p,\mu) \right)^p
	\end{equation}
	where
	\[ \Delta(\eps,p,\mu) = \begin{cases}
		\left[ \left( \energy{1,p}(\proj{\mu}{j}) + \energy{1,p}(\proj{\mu}{j} * \eta^\eps)\right)^{\frac{1}{p-1}} - 2 \energy{1,p}(\proj{\mu}{j} * \eta^\eps)^{\frac{1}{p-1}} \right]^\frac{p-1}{p} & 1 < p < 2 \\
		\left( \energy{1,p}(\proj{\mu}{j}) - \energy{1,p}(\proj{\mu}{j} * \eta^\eps)\right)^\frac{1}{p} & p \geq 2
	\end{cases} \]
	and $c_p$ is a suitable constant depending only on the exponent $p$.
\end{thm}

\begin{proof}
	Combining \autoref{thm-regularity} and \autoref{nabla-domination} we get the first part of the statement, proceeding as in the proof of \autoref{energy-estimate} and using the triangular inequality in $L^p$. When the marginals are regular, we use \eqref{gradient-formula} to write $\nabla_{x_j} \Theta^\eps[\mu] (X) = I(X) + II(X)$, and estimate both terms via the H\"{o}lder inequality to get
	\begin{ieee*}{rCl}
		\abs{I(X)}^p & \leq & \Theta^\eps[\mu](X)^{p-1} \int \abs{\frac{\nabla \proj{\mu}{j}(x_j)}{\proj{\mu}{j}(x_j)} - \frac{\nabla (\proj{\mu}{j} * \eta^\eps)(y_j)}{(\proj{\mu}{j} * \eta^\eps)(y_j)}}^p P^\eps[\mu](X, Y) \dY, \\
		\abs{II(X)}^p & \leq & \Theta^\eps[\mu](X)^{p-1} \int \frac{\abs{\nabla_{x_j} \Lambda^\eps[\mu](Y)}^p}{\Lambda^\eps[\mu](Y)^p} P^\eps[\mu](X,Y) \dY.
	\end{ieee*}
	
	When we integrate with respect to the $X$ variable, the triangular inequality in $L^p$ gives
	\begin{ieee*}{Cl}
		\IEEEeqnarraymulticol{2}{l}{\left( \int \frac{\abs{\nabla_{x_j} \Theta^\eps[\mu] (X)}^p}{\Theta^\eps[\mu](X)^{1-p}} \dX \right)^\frac{1}{p}} \\
		\leq & \left( \iint \abs{\frac{\nabla \proj{\mu}{j}(x_j)}{\proj{\mu}{j}(x_j)} - \frac{\nabla (\proj{\mu}{j} * \eta^\eps)(y_j)}{(\proj{\mu}{j} * \eta^\eps)(y_j)}}^p P^\eps[\mu](X, Y) \dY \dX \right)^{\frac{1}{p}} \\
		& + \left( \iint \frac{\abs{\nabla_{x_j} \Lambda^\eps[\mu](Y)}^p}{\Lambda^\eps[\mu](Y)^p} P^\eps[\mu](X,Y) \dY \dX \right)^\frac{1}{p} \\
		= & \left( \iint \abs{\frac{\nabla \proj{\mu}{j}(x)}{\proj{\mu}{j}(x)} - \frac{\nabla (\proj{\mu}{j} * \eta^\eps)(y)}{(\proj{\mu}{j} * \eta^\eps)(y)}}^p \eta^\eps(y-x) \proj{\mu}{j}(x) \de x \de y \right)^{\frac{1}{p}} \\
		& + \left( \int \frac{\abs{\nabla_{x_j} \Lambda^\eps[\mu](Y)}^p}{\Lambda^\eps[\mu](Y)^{p-1}} \dY \right)^\frac{1}{p}. \\
	\end{ieee*}
	
	Now we recall the following inequalities by Clarkson \cite{clarkson1936uniformly}: if $f,g \in L^p(\nu)$, then
	\begin{ieee}{rCll} 
		\norm{\frac{f-g}{2}}^p & \leq & \frac{1}{2} \norm{f}^p + \frac{1}{2} \norm{g}^p - \norm{\frac{f+g}{2}}^p \quad & p \geq 2 \label{clarkson-1} \\
		\norm{\frac{f-g}{2}}^\frac{p}{p-1} & \leq & \left( \frac{1}{2} \norm{f}^p + \frac{1}{2}\norm{g}^p \right) ^{\frac{1}{p-1}} - \norm{\frac{f+g}{2}}^\frac{p}{p-1} \quad & 1 < p < 2, \label{clarkson-2}
	\end{ieee}
	where all the norms are $L^p(\nu)$ norms.
	
	If we apply \eqref{clarkson-1} on $\R^d \times \R^d$ with $f(x,y) = \frac{\nabla \proj{\mu}{j}(x)}{\proj{\mu}{j}(x)}$, $g(x,y) = \frac{\nabla (\proj{\mu}{j} * \eta^\eps)(y)}{(\proj{\mu}{j} * \eta^\eps) (y)}$ and $\frac{\de \nu}{\de \Leb^d}(x,y) = \eta^\eps(y-x) \proj{\mu}{j}(x)$, we get for $p \geq 2$
	\begin{ieee*}{Cl}
		\IEEEeqnarraymulticol{2}{l}{\iint \abs{\frac{\nabla \proj{\mu}{j}(x)}{\proj{\mu}{j}(x)} - \frac{\nabla (\proj{\mu}{j} * \eta^\eps)(y)}{(\proj{\mu}{j} * \eta^\eps)(y)}}^p \eta^\eps(y-x) \proj{\mu}{j}(x) \de x \de y} \\
		\leq & 2^{p-1} \iint \abs{\frac{\nabla \proj{\mu}{j}(x)}{\proj{\mu}{j}(x)}}^p \eta^\eps(y-x) \proj{\mu}{j}(x) \de x \de y \\
		& {} + 2^{p-1} \iint \abs{\frac{\nabla (\proj{\mu}{j} * \eta^\eps)(y)}{(\proj{\mu}{j} * \eta^\eps)(y)}}^p \eta^\eps(y-x) \proj{\mu}{j}(x) \de x \de y \\
		& {} - \iint \abs{\frac{\nabla \proj{\mu}{j}(x)}{\proj{\mu}{j}(x)} + \frac{\nabla (\proj{\mu}{j} * \eta^\eps)(y)}{(\proj{\mu}{j} * \eta^\eps)(y)}}^p \eta^\eps(y-x) \proj{\mu}{j}(x) \de x \de y \\
		=  & 2^{p-1} \int \frac{\abs{\nabla \proj{\mu}{j}(x)}^p}{\proj{\mu}{j}(x)^{p-1}} \de x + 2^{p-1} \int \frac{\abs{\nabla (\proj{\mu}{j} * \eta^\eps)(y)}^p}{(\proj{\mu}{j} * \eta^\eps)(y)^{p-1}} \de y \\
		& {} - \iint \abs{\frac{\nabla \proj{\mu}{j}(x)}{\proj{\mu}{j}(x)} + \frac{\nabla (\proj{\mu}{j} * \eta^\eps)(y)}{(\proj{\mu}{j} * \eta^\eps)(y)}}^p \eta^\eps(y-x) \proj{\mu}{j}(x) \de x \de y.
	\end{ieee*}
	
	On the other hand, using \eqref{clarkson-2}, for $1 < p < 2$ we have
	\begin{ieee*}{Cl}
		\IEEEeqnarraymulticol{2}{l}{\iint \abs{\frac{\nabla \proj{\mu}{j}(x)}{\proj{\mu}{j}(x)} - \frac{\nabla (\proj{\mu}{j} * \eta^\eps)(y)}{(\proj{\mu}{j} * \eta^\eps)(y)}}^p \eta^\eps(y-x) \proj{\mu}{j}(x) \de x \de y} \\
		\leq & \bigg[ \left( 2^{p-1} \int \frac{\abs{\nabla \proj{\mu}{j}(x)}^p}{\proj{\mu}{j}(x)^{p-1}} \de x + 2^{p-1} \int \frac{\abs{\nabla (\proj{\mu}{j} * \eta^\eps)(y)}^p}{(\proj{\mu}{j} * \eta^\eps)(y)^{p-1}} \de y \right)^{\frac{1}{p-1}} \\
		& {} - \left( \iint \abs{\frac{\nabla \proj{\mu}{j}(x)}{\proj{\mu}{j}(x)} + \frac{\nabla (\proj{\mu}{j} * \eta^\eps)(y)}{(\proj{\mu}{j} * \eta^\eps)(y)}}^p \eta^\eps(y-x) \proj{\mu}{j}(x) \de x \de y \right)^{\frac{1}{p-1}} \bigg]^{p-1}.
	\end{ieee*}
	
	Finally, by convexity of the function $z \mapsto \abs{z}^p$ on $\R^d$ we have
	\begin{ieee*}{Cl}
		\IEEEeqnarraymulticol{2}{l}{\iint \abs{\frac{\nabla \proj{\mu}{j}(x)}{\proj{\mu}{j}(x)} + \frac{\nabla (\proj{\mu}{j} * \eta^\eps)(y)}{(\proj{\mu}{j} * \eta^\eps)(y)}}^p \eta^\eps(y-x) \proj{\mu}{j}(x) \de x \de y} \\
		\geq & \iint \abs{\frac{2\nabla (\proj{\mu}{j} * \eta^\eps)(y)}{(\proj{\mu}{j} * \eta^\eps)(y)}}^p \eta^\eps(y-x) \proj{\mu}{j}(x) \de x \de y \\
		& {} + 2^{p-1} p \sum_{j = 1}^d \iint \abs{\frac{\partial_j (\proj{\mu}{j} * \eta^\eps)(y)}{(\proj{\mu}{j} * \eta^\eps)(y)}}^{p-2} \frac{\partial_j (\proj{\mu}{j} * \eta^\eps)(y)}{(\proj{\mu}{j} * \eta^\eps)(y)} \frac{\partial_j \proj{\mu}{j}(x)}{\proj{\mu}{j}(x)} \eta^\eps(y-x) \proj{\mu}{j}(x) \de x \de y \\
		& {} - 2^{p-1} p \sum_{j = 1}^d \iint \abs{\frac{\partial_j (\proj{\mu}{j} * \eta^\eps)(y)}{(\proj{\mu}{j} * \eta^\eps)(y)}}^{p-2} \frac{\abs{\partial_j (\proj{\mu}{j} * \eta^\eps)(y)}^2}{(\proj{\mu}{j} * \eta^\eps)(y)^2} \eta^\eps(y-x) \proj{\mu}{j}(x) \de x \de y \\
		= & 2^p \int \frac{\abs{\nabla (\proj{\mu}{j} * \eta^\eps)(y)}^p}{(\proj{\mu}{j} * \eta^\eps)(y)^{p-1}} \de y.
	\end{ieee*}
	
	Hence, for $p \geq 2$,
	\begin{ieee*}{Cl}
		\IEEEeqnarraymulticol{2}{l}{\iint \abs{\frac{\nabla \proj{\mu}{j}(x)}{\proj{\mu}{j}(x)} - \frac{\nabla (\proj{\mu}{j} * \eta^\eps)(y)}{(\proj{\mu}{j} * \eta^\eps)(y)}}^p \eta^\eps(y-x) \proj{\mu}{j}(x) \de x \de y} \\
		\leq & 2^{p-1} \left( \int \frac{\abs{\nabla \proj{\mu}{j}(x)}^p}{\proj{\mu}{j}(x)^{p-1}} \de x - \int \frac{\abs{\nabla (\proj{\mu}{j} * \eta^\eps)(y)}^p}{(\proj{\mu}{j} * \eta^\eps)(y)^{p-1}} \de y \right) \\
		= & 2^{p-1} p^p \left( \energy{1,p}(\proj{\mu}{j}) - \energy{1,p}(\proj{\mu}{j} * \eta^\eps)\right),
	\end{ieee*}
	while for $1 < p < 2$
	\begin{ieee*}{Cl}
		\IEEEeqnarraymulticol{2}{l}{\iint \abs{\frac{\nabla \proj{\mu}{j}(x)}{\proj{\mu}{j}(x)} - \frac{\nabla (\proj{\mu}{j} * \eta^\eps)(y)}{(\proj{\mu}{j} * \eta^\eps)(y)}}^p \eta^\eps(y-x) \proj{\mu}{j}(x) \de x \de y} \\
		\leq & 2^{p-1} p^p \left[ \left( \energy{1,p}(\proj{\mu}{j}) + \energy{1,p}(\proj{\mu}{j} * \eta^\eps)\right)^{\frac{1}{p-1}} - 2 \energy{1,p}(\proj{\mu}{j} * \eta^\eps)^{\frac{1}{p-1}} \right]^{p-1},
	\end{ieee*}
	
	Putting all together and summing on $j$ we get the thesis with $c_p = 2^{\frac{p-1}{p}}$.
\end{proof}

\begin{remark} \label{energy-remark}
	As one would expect, if the measure $\mu$ is not regular then the bound on the energy of $\Theta^\eps[\mu]$ diverges as $\eps$ approaches zero, as in \eqref{energy-bound-1} and \eqref{p-energy-bound-1}. On the contrary, if $\mu$ is $W^{1,p}$-regular then the bound on the energy of $\Theta^\eps[\mu]$ in \eqref{energy-bound-2} and \eqref{p-energy-bound-2} converges to the energy of $\mu$ as $\eps \to 0$. Indeed, on the one hand $\Delta(\eps,p,\mu)$ converges to zero by \autoref{energy-convergence}. On the other hand, let $\lambda^\eps(z_1, \dotsc, z_N) = \eta^\eps(z_1) \dotsm \eta^\eps(z_N)$, we have $\Lambda^\eps[\mu] = \mu * \lambda^\eps$, and hence 
	\[ \norm[p]{\nabla_{x_j} (\mu * \lambda^\eps)^\frac{1}{p}} \to \norm[p]{\nabla_{x_j} \mu^\frac{1}{p}}. \]
	
	When we raise to the power $p$ and sum over $j$ we get $\energy{1,p}(\mu)$in view of the usual condition \eqref{abs-gradient-convention}.
\end{remark}


\section{Continuity of $\Theta$ in $\eps$} \label{section-eps-continuity}

Finally, in this section we prove that $\Theta$ satisfies property \ref{Theta-Cb-continuity} of Section \ref{section-definition}. In order to simplify the notation, let as above $P^\eps[\mu]$ be the measure over $\Rdn \times \Rdn$ given by
\[ \iint \psi(X,Y) \de P^\eps[\mu](X,Y) \eqdef \iint \psi(X,Y) \prod_{k = 1}^N \frac{\eta^\eps(y_k-x_k)}{(\proj{\mu}{k} * \eta^\eps) (y_k)} \de \proj{\mu}{k} (x_k) \de \Lambda^\eps[\mu](Y), \]
already introduced in \autoref{section-regularity}, and let $Q^\eps[\mu]$ be the measure over $\Rdn \times \Rdn$ given by
\[ \iint \psi(X,Y) \de Q^\eps[\mu](X,Y) \eqdef \iint \psi(X,Y) \prod_{k = 1}^N \eta^\eps(y_k-x_k) \de \mu(X) \dY \]
for any $\psi \colon \Rdn \times \Rdn \to \R$ bounded and countinuous. 

\begin{remark} \label{P-Q-remark}
	Notice that, if $\psi \in C_b\left( \Rdn \right)$, then recalling definitions \eqref{Lambda-definition} and \eqref{Theta-definition} we have
	\begin{ieee*}{rCl}
		\iint \psi(X) \de P^\eps[\mu](X,Y) & = & \iint \psi(X) \prod_{k = 1}^N \frac{\eta^\eps(y_k-x_k)}{(\proj{\mu}{k} * \eta^\eps) (y_k)} \de \proj{\mu}{k} (x_k) \de \Lambda^\eps[\mu](Y) \\
		& = & \int \psi(X) \de \Theta^\eps[\mu](X),
	\end{ieee*}
	while
	\begin{ieee*}{rCl}
		\iint \psi(Y) \de P^\eps[\mu](X,Y) & = & \iint \psi(Y) \prod_{k = 1}^N \frac{\eta^\eps(y_k-x_k)}{(\proj{\mu}{k} * \eta^\eps) (y_k)} \de \proj{\mu}{k} (x_k) \de \Lambda^\eps[\mu](Y) \\
		& = & \int \psi(Y) \de \Lambda^\eps[\mu](Y).
	\end{ieee*}
	
	On the other hand,
	\begin{ieee*}{rCl}
		\iint \psi(X) \de Q^\eps[\mu](X,Y) & = & \iint \psi(X) \prod_{k = 1}^N \eta^\eps(y_k-x_k) \de \mu(X) \dY \\
		& = & \int \psi(X) \de \mu(X),
	\end{ieee*}
	while
	\begin{ieee*}{rCl}
		\iint \psi(Y) \de Q^\eps[\mu](X,Y) & = & \iint \psi(Y) \prod_{k = 1}^N \eta^\eps(y_k-x_k) \de \mu(X) \dY \\
		& = & \int \psi(Y) \de \Lambda^\eps[\mu](Y).
	\end{ieee*}
\end{remark}

Let us introduce a couple of technical results.

\begin{lemma} \label{eta-estimate} There exists a constant $K(d)$, depending only on the dimension $d$, such that for every $\eps, \tau > 0$,
	\[ \int_{\gra{\abs{z} \geq \tau}} \eta^\eps(z) \de z \leq K(d) e^{-\frac{\tau^2}{4\eps}}. \]
\end{lemma}

\begin{proof} It is just a computation: passing to spherical coordinates and denoting by $\sigma_d$ the surface area of the unit sphere in $\R^d$,
	\begin{align*}
	\int_{\gra{\abs{z} \geq \tau}} \eta^\eps(z) \de z &= \frac{\sigma_d}{(2\pi\eps)^{\frac{d}{2}}} \int_\tau^{+\infty} r^{d-1} e^{-\frac{r^2}{2\eps}} \de r \\
	& = \frac{\sigma_d}{2 \pi^{\frac{d}{2}}} \int_{\frac{\tau^2}{2 \eps}}^{+\infty} s^{\frac{d-2}{2}} e^{-s} \de s \\
	& \leq \frac{\sigma_d}{2 \pi^{\frac{d}{2}}} e^{-\frac{\tau^2}{4\eps}} \int_{\frac{\tau^2}{2 \eps}}^{+\infty} s^{\frac{d-2}{2}} e^{-\frac{s}{2}} \de s \\
	& \leq \frac{\sigma_d}{2 \pi^{\frac{d}{2}}} e^{-\frac{\tau^2}{4\eps}} \int_0^{+\infty} s^{\frac{d-2}{2}} e^{-\frac{s}{2}} \de s \\
	& = \frac{\sigma_d}{2} \left(\frac{2}{\pi}\right)^{\frac{d}{2}} \Gamma \left( \frac{d}{2} \right) e^{-\frac{\tau^2}{4\eps}}.
	\tag*{\qedhere}
	\end{align*}
\end{proof}

\begin{lemma} \label{diagonal-estimate}
	For every $r, \eps > 0$ and for every $\mu \in \prob\left( \Rdn \right)$,
	\begin{ieee*}{rCl}
		P^\eps[\mu]\left( \gra{\abs{X-Y} \geq r} \right) & \leq & NK(d) \exp \left( -\frac{r^2}{4N\eps} \right) \\
		Q^\eps[\mu]\left( \gra{\abs{X-Y} \geq r} \right) & \leq & NK(d) \exp \left( -\frac{r^2}{4N\eps} \right),
	\end{ieee*}
	where $K(d)$ is the constant in \autoref{eta-estimate}.
\end{lemma}

\begin{proof} 
	
	Observe that 
	\begin{ieee*}{l}
		\gra{(X,Y) \in \Rdn \times \Rdn \st \abs{X-Y} \geq r} \\
		\subseteq \bigcup_{j = 1}^N \gra{(X,Y) \in \Rdn \times \Rdn \st \abs{x_j-y_j} \geq \frac{r}{\sqrt{N}}}.
	\end{ieee*}
	
	Using \autoref{eta-estimate}, this yields
	\begin{align*}
	P^\eps[\mu](\gra{\abs{X-Y} \geq r}) &\leq \sum_{j = 1}^N P^\eps[\mu]\left( \gra{\abs{x_j-y_j} \geq \frac{r}{\sqrt{N}}} \right) \\
	& = \sum_{j = 1}^N \int_{\gra{\abs{x_j-y_j} \geq \frac{r}{\sqrt{N}}}} \prod_{k = 1}^N \frac{\eta^\eps(y_k-x_k)}{(\proj{\mu}{k} * \eta^\eps) (y_k)} \de \proj{\mu}{k} (x_k) \de \Lambda^\eps[\mu](Y) \\
	& = \sum_{j = 1}^N \int_{\gra{\abs{x_j-y_j} \geq \frac{r}{\sqrt{N}}}} \eta^\eps(y_j-x_j) \de \proj{\mu}{j}(x_j) \de y_j \\
	& = \sum_{j = 1}^N \int_{\gra{\abs{z_j} \geq \frac{r}{\sqrt{N}}}} \eta^\eps(z_j) \de \proj{\mu}{j}(x_j) \de z_j \\
	& = N \int_{\gra{\abs{z} \geq \frac{r}{\sqrt{N}}}} \eta^\eps(z) \de z \leq N K(d) \exp \left( -\frac{r^2}{4N\eps} \right).
	\end{align*}
	
	Analogously, 
	\begin{align*}
	Q^\eps[\mu](\gra{\abs{X-Y} \geq r}) &\leq \sum_{j = 1}^N Q^\eps[\mu]\left( \gra{\abs{x_j-y_j} \geq \frac{r}{\sqrt{N}}} \right) \\
	& = \sum_{j = 1}^N \int_{\gra{\abs{x_j-y_j} \geq \frac{r}{\sqrt{N}}}} \prod_{k = 1}^N \eta^\eps(y_k-x_k) \de \mu (X) \dY \\
	& = \sum_{j = 1}^N \int_{\gra{\abs{x_j-y_j} \geq \frac{r}{\sqrt{N}}}} \eta^\eps(y_j-x_j) \de \proj{\mu}{j}(x_j) \de y_j \\
	& = \sum_{j = 1}^N \int_{\gra{\abs{z_j} \geq \frac{r}{\sqrt{N}}}} \eta^\eps(z_j) \de \proj{\mu}{j}(x_j) \de z_j \\
	& = N \int_{\gra{\abs{z} \geq \frac{r}{\sqrt{N}}}} \eta^\eps(z) \de z \leq N K(d) \exp \left( -\frac{r^2}{4N\eps} \right).
	\tag*{\qedhere}
	\end{align*}
\end{proof}

We now move towards the proof of property \ref{Theta-Cb-continuity}. Even though it requires to test the convergence of $\Theta^\eps[\mu]$ to $\mu$ for all the continuous and bounded functions, first we prove the convergence for a smaller class, namely the continuous functions with compact support.

\begin{prop} \label{Cc-duality} Let $\mu \in \prob\left( \Rdn \right)$. Then, for every $\psi \in C_c\left( \Rdn \right)$,
	\[  \lim_{\eps \to 0} \int \psi(X) \de \Theta^\eps[\mu](X) = \int \psi(X) \de \mu(X). \]
\end{prop}

\begin{proof}
	Fix $\psi \colon \Rdn \to \R$  a continuous function with compact support and $\delta > 0$. Since $\psi$ is absolutely continuous, let $\eps_0 > 0$ be such that
	\[ \abs{X-Y} < \eps_0^\frac{1}{4} \implies \abs{\psi(X) - \psi(Y)} < \delta. \]
	
	Using \autoref{P-Q-remark} we have:
	\begin{ieee*}{Cl}
		\IEEEeqnarraymulticol{2}{l}{\abs{ \int \psi(X) \de \Theta^\eps[\mu](X) - \int \psi(X) \de \mu(X)}} \\
		\leq & \abs{\int \psi(X) \de \Theta^\eps[\mu](X) - \int \psi(Y) \de \Lambda^\eps[\mu](Y)} \\
		&{} + \abs{\int \psi(Y) \de \Lambda^\eps[\mu](Y) - \int \psi(X) \de \mu(X)} \\
		\leq & \iint \abs{\psi(X) - \psi(Y)} \de P^\eps[\mu](X,Y) + \iint \abs{\psi(Y) - \psi(X)} \de Q^\eps[\mu](X,Y). \\
	\end{ieee*}
	
	Let us put
	\begin{ieee*}{rCl}
		A^\eps & = & \gra{(X,Y) \in \Rdn \times \Rdn \colon \abs{X-Y} \geq \eps^\frac{1}{4}} \\
		B^\eps & = & \gra{(X,Y) \in \Rdn \times \Rdn \colon \abs{X-Y} < \eps^\frac{1}{4}}.
	\end{ieee*}
	
	Using \autoref{diagonal-estimate},
	\begin{ieee*}{rCl}
		\iint_{A^\eps} \abs{\psi(X) - \psi(Y)} \de P^\eps[\mu](X,Y) & \leq & 2 \norm[\infty]{\psi} P^\eps[\mu] \left( A^\eps \right) \\
		& \leq & 2 NK(d) \norm[\infty]{\psi} \exp \left( -\frac{1}{4N\sqrt{\eps}} \right).
	\end{ieee*}
	which goes to zero as $\eps \to 0$. On the other hand, for every $\eps < \eps_0$ we have
	\begin{ieee*}{rCl}
		\iint_{B^\eps} \abs{\psi(X) - \psi(Y)} \de P^\eps[\mu](X,Y) & \leq & \delta P^\eps[\mu](B^\eps) \leq \delta.
	\end{ieee*}
	
	Treating the integral with respect to the measure $Q^\eps[\mu]$ in the same way we get the thesis since $\delta$ was arbitrary.
\end{proof} 

One way to extend the result of \autoref{Cc-duality} to the continuous and bounded functions is to use the Prokhorov's theorem (\autoref{prokhorov}), by first proving that, for every $\mu \in \prob\left( \Rdn \right)$, the family $\gra{\Theta^\eps(\mu)}_{\eps > 0}$ is tight. In view of \autoref{basic-properties}, this is actually a simple corollary of the following more general result.

\begin{thm}
	Let $\mathcal{M} \subseteq \prob\left( \Rdn \right)$ such that, for every $\mu, \nu \in \mathcal{M}$ and every $j = 1, \dotsc, N$,
	\[ \proj{\mu}{j} = \proj{\nu}{j}. \]
	Then $\mathcal{M}$ is tight. 
\end{thm}

\begin{proof}
	Let $\rho_1, \dotsc, \rho_N$ be the common marginals of all the measures in $\mathcal{M}$, and fix $\delta > 0$. Since every $\rho_j$ is a probability, we may find $K \subseteq \R^d$ compact such that $\rho_j(K) \geq 1 - \frac{\delta}{N}$ for all $j = 1, \dotsc, N$. Let $K^N \eqdef K \times \dotsb \times K \subseteq \Rdn$, which is compact. We claim that $\mu(K^N) \geq 1 - \delta$ for all $\mu \in \mathcal{M}$. First notice that
	\[ \left( K^N \right) ^c = \bigcup_{j = 1}^N \big( \R^d \times \dotsb \times \underset{\stackrel{\uparrow}{j\text{-th}}}{K^c} \times \dotsb \times \R^d \big). \]
	
	Hence, for every $\mu \in \mathcal{M}$,
	\begin{ieee*}{rCl}
		\mu \left( (K^N)^c \right) & \leq & \sum_{j = 1}^N \mu \big( \R^d \times \dotsb \times \underset{\stackrel{\uparrow}{j\text{-th}}}{K^c} \times \dotsb \times \R^d \big) \\
		& = & \sum_{j = 1}^N \proj{\mu}{j}(K^c) = \sum_{j = 1}^N \rho_j(K^c) \leq \sum_{j = 1}^N \frac{\delta}{N} = \delta,
	\end{ieee*}	
	
	so that $\mu(K^N) \geq 1 - \delta$.
\end{proof}

Finally combining \autoref{prokhorov} with \autoref{Cc-duality} we get the convergence of $\Theta^\eps[\mu]$ to $\mu$ in duality with $C_b\left( \Rdn \right)$, as wanted.

\begin{thm} \label{Cb-duality} Let $\mu \in \prob\left( \Rdn \right)$. Then, for every $\psi \in C_b\left( \Rdn \right)$,
	\[  \lim_{\eps \to 0} \int \psi(X) \de \Theta^\eps[\mu](X) = \int \psi(X) \de \mu(X). \]
\end{thm}

\begin{proof}
	Suppose by contradiction that there exists $\delta > 0$, a sequence $\eps_n \searrow 0$ and a continuous bounded function $\psi \colon \Rdn \to \R$ such that
	\begin{equation} \label{Cb-absurd}
	\abs{\int \psi(X) \de \Theta^{\eps_n}[\mu](X) - \int \psi(X) \de \mu(X)} \geq \delta > 0.
	\end{equation}
	
	Denote for simplicity $\mu_n \eqdef \Theta^{\eps_n}[\mu]$. We know that the family $\gra{\mu_n}_{n \in \N}$ is tight, and by \autoref{prokhorov} we may extract a subsequence $\mu_{n_k}$ weakly converging to some $\nu \in \prob\left( \Rdn \right)$. However \autoref{Cc-duality} ensures that $\nu = \mu$, and hence $\mu_{n_k} \wconv \mu$, contradicting \eqref{Cb-absurd}. 
\end{proof}

We conclude this section with a final result about the continuity of $\Theta$. We proved in \autoref{Cb-duality} that $\Theta^\eps[\mu] \wconv \mu$ as $\eps \to 0$, which is the natural notion of convergence as far as $\mu$ is no more regular than a measure. However if $\mu$ has some better regurality, say $\mu \in \prob^{1,p}\left( \Rdn \right)$, since $\Theta^\eps[\mu] \in \prob^{1,p}$ for every $\eps > 0$ it is natural to ask whether $\Theta^\eps[\mu] \to \mu$ in the $d^{1,p}$-topology. The answer is positive, as stated in the following

\begin{thm} Let $\mu \in \prob^{1,p}\left( \Rdn \right)$, with $p > 1$. Then
	\[ \lim_{\eps \to 0} d^{1,p}(\Theta^\eps[\mu], \mu) = 0. \]
\end{thm}

\begin{proof} Combining the fact that the family $\Theta^\eps[\mu]^\frac{1}{p}$ is bounded in $W^{1,p}$ due to \autoref{p-energy-estimate} and the result of \autoref{Cb-duality} we get that $\Theta^\eps[\mu]^\frac{1}{p} \to \mu^\frac{1}{p}$ weakly in $W^{1,p}\left( \Rdn \right)$ as $\eps \to 0$. Since $W^{1,p}$ is uniformly convex, we need only to check that
	\[ \lim_{\eps \to 0} \norm[W^{1,p}]{\Theta^\eps[\mu]^\frac{1}{p}} = \norm[W^{1,p}]{\mu^\frac{1}{p}}. \]
	
	The $L^p$-norms are identically equal to 1, so we need to prove the limit for the norms of the gradients. The weak convergence of $\nabla \Theta^\eps[\mu]^\frac{1}{p}$ to $\nabla \mu^\frac{1}{p}$ implies that
	\[
	\liminf_{\eps \to 0}\, \norm[L^p]{\nabla \Theta^\eps[\mu]^\frac{1}{p}} \geq \norm[L^p]{\nabla \mu^\frac{1}{p}}.
	\]
	
	The other inequality follows from \autoref{energy-remark}.
\end{proof}


\section{Continuity of $\Theta$ in $\mu$} \label{section-mu-continuity}

We devote this final section to the proof of \autoref{mu-continuity}. Throughout this section, $\eps$ will be fixed and positive. The main idea for the proof of \autoref{mu-continuity} is to use Lebesgue's dominated convergence theorem, but in order to do so we must have some fine upper-bound on the integral kernel $P^\eps[\mu]$ defining $\Theta^\eps[\mu]$. We refer to \eqref{Lambda-definition} and \eqref{Theta-definition} for the definitions. With a slight abuse of notation, since $\Lambda^\eps[\mu]$ and $\proj{\mu}{j} * \eta^\eps$ are absolutely continuous with respect to the Lebesgue measure, we will use the same symbol for the measure and its density.

\begin{lemma} \label{upper-lower-lemma} Let $\mu \in \prob\left( \Rdn \right)$. Then:
	\begin{enumerate}[(i)]
		\item 
		\[
		\Lambda^\eps[\mu](Y) \leq (2\pi\eps)^{-\frac{(N-1)d}{2N}} \prod_{k = 1}^N (\proj{\mu}{k} * \eta^\eps)(y_k)^\frac{1}{N}.
		\]
		\item Let $R > 0$, $\gamma \in [0,1]$ be such that $\proj{\mu}{j}(B(0,R)) \geq \gamma$. Then
		\[
		(\proj{\mu}{j} * \eta^\eps) (y_j) \geq \frac{\gamma}{(2\pi\eps)^{d/2}} \exp \left( -\frac{(\abs{y_j} + R)^2}{2\eps} \right).
		\]
	\end{enumerate}
	
\end{lemma}

\begin{proof} (i) We apply a general version of the H\"older's inequality with exponents $p_1 = \dotsb p_N = N$, and use the fact that $\eta^\eps(z) \leq \eta^\eps(0) = (2\pi\eps)^{-d/2}$, to get
	
	\begin{ieee*}{rCl}
		\Lambda^\eps[\mu](Y) & = & \int \prod_{j = 1}^N \eta^\eps(y_j-x_j) \de \mu(X) \\
		& \leq & \prod_{j = 1}^N \left( \int \eta^\eps(y_j-x_j)^N \de \mu(X) \right)^\frac{1}{N} \\
		& \leq & (2\pi\eps)^{-\frac{(N-1)d}{2N}} \prod_{j = 1}^N \left( \int \eta^\eps(y_j-x_j) \de \mu(X) \right)^\frac{1}{N} \\
		& = & (2\pi\eps)^{-\frac{(N-1)d}{2N}} \prod_{j = 1}^N (\proj{\mu}{j} * \eta^\eps)(y_j)^\frac{1}{N}.
	\end{ieee*}
	as wanted.
	
	\bigskip
	
	(ii) We start observing that
	\begin{align*}
	(\proj{\mu}{j} * \eta^\eps)(y_j) = \int \eta^\eps(y_j-x_j) \de \proj{\mu}{j}(x_j) \geq \int_{B(0,R)} \eta^\eps(y_j-x_j) \de \proj{\mu}{j}(x_j).
	\end{align*}
	
	When $x_j$ belongs to the ball $B(0,R)$, the minimum value of $\eta^\eps(y_j-x_j)$ is attained at $x_j = -R\frac{y_j}{\abs{y_j}}$, or at any boundary point if $y_j = 0$. Thus, in this region,
	\[
	\eta^\eps(y_j-x_j) \geq \frac{1}{(2\pi\eps)^{d/2}} \exp \left( -\frac{(\abs{y_j} + R)^2}{2\eps} \right)
	\]
	and the thesis follows easily.
\end{proof}

\begin{lemma} \label{tightness-lemma}
	Let $\rho_n, \rho \in \prob^{1,p}(\R^d)$ such that $\rho_n \to \rho$ in the $d^{1,p}$-topology. Then the family $\gra{\rho} \cup \gra{\rho_n}_{n \in \N}$ is tight. In particular, for every $\gamma > 0$ there exists $R > 0$ such that $\rho_n(B(0,R)) \geq 1 - \gamma$ and $\rho(B(0,R)) \geq 1 - \gamma$.
\end{lemma}

\begin{proof}
	Due to Prokhorov's theorem (\autoref{prokhorov}), it suffices to show that $\rho_n \wconv \rho$. However, by \autoref{L1-convergence-marginals} we have the stronger property $\rho_n \to \rho$ in $W^{1,1}(\R^d)$.
\end{proof}

\begin{prop} \label{pointwise-prop} Suppose that $\mu^n \wconv \mu$, with $\proj{\mu^n}{j} \to \proj{\mu}{j}$ in $\prob^{1,p}(\R^d)$ and $\proj{\mu^n}{j} \to \proj{\mu}{j}$ pointwise a.e. on $\R^d$ for every $j = 1, \dotsc, N$. Then $\Theta^\eps[\mu^n] \to \Theta^\eps[\mu]$ pointwise a.e. on $\Rdn$.
	
	Assume in addiction that $\nabla \proj{\mu^n}{j} \to \nabla \proj{\mu}{j}$ pointwise a.e. on $\R^d$. Then $\nabla \Theta^\eps[\mu^n] \to \nabla \Theta^\eps[\mu]$ pointwise a.e. on $\Rdn$.
\end{prop}

\begin{proof} Let $P^\eps[\mu](X,Y)$ be the integral kernel defining $\Theta^\eps[\mu]$, namely
	\[  P^\eps[\mu](X,Y) = \prod_{j = 1}^N \frac{\eta^\eps(y_j-x_j)}{(\proj{\mu}{j} * \eta^\eps)(y_j)} \proj{\mu}{j}(x_j) \Lambda^\eps[\mu](Y). \]	
	
	We claim that $P^\eps[\mu^n]$ converges pointwise a.e. to $P^\eps[\mu]$. For every $Y \in \Rdn$ and every $j \in \gra{1, \dotsc, N}$ we have
	\begin{ieee*}{rCl}
		\abs{(\proj{\mu^n}{j} * \eta^\eps)(y_j) - (\proj{\mu}{j} * \eta^\eps)(y_j)} & \leq & \int \eta^\eps(y_j - x_j) \abs{\proj{\mu^n}{j}(x_j) - \proj{\mu}{j}(x_j)} \de x_j	\\
		& \leq & \frac{1}{(2\pi\eps)^\frac{d}{2}} \norm[1]{\proj{\mu^n}{j} - \proj{\mu}{j}} \to 0
	\end{ieee*}	
	by \autoref{p-continuous-maps}. Moreover 
	\begin{ieee*}{l}
		\abs{\Lambda^\eps[\mu^n](Y) - \Lambda^\eps[\mu](Y)} \\
		\leq \abs{\int \prod_{j = 1}^N \eta^\eps(y_j-x_j) \de \mu^n(X) - \int \prod_{j = 1}^N \eta^\eps(y_j-x_j) \de \mu(X)}
	\end{ieee*}
	goes to zero for every $Y$ because $\prod \eta^\eps(y_j-x_j)$ is a fixed countinuous bounded function, and $\mu^n \wconv \mu$.  Finally fix $X \in \Rdn$ in the set of full measure such that $\proj{\mu^n}{j}(x_j) \to \proj{\mu}{j}(x_j)$ for every $j = 1, \dotsc, N$.
	
	We need only to find a domination for $P^\eps[\mu^n]$. For every $j = 1, \dotsc, N$ let $R_j$ given by \autoref{tightness-lemma} for $\gamma = \frac{1}{2}$, and let $R = \max_{j} R_j$. Using \autoref{upper-lower-lemma} (i) and (ii) one has
	\begin{align*}
	P^\eps[\mu^n](X,Y) &\leq (2\pi\eps)^{-\frac{(N-1)d}{2N}} \prod_{j = 1}^N \frac{\eta^\eps(y_j-x_j) \proj{\mu^n}{j}(x_j)}{(\proj{\mu^n}{j} * \eta^\eps)(y_j)^{N-1/N}}\\
	&\leq 2^N \prod_{j = 1}^N \eta^\eps(y_j-x_j) \proj{\mu^n}{j}(x_j) \exp \left( \frac{(N-1)(\abs{y_j} + R)^2}{2N\eps} \right) \\
	&= 2^N e^{\frac{(N-1)R^2}{2\eps}} \prod_{j = 1}^N \proj{\mu^n}{j}(x_j) e^{\frac{-|x_j|^2}{2\eps}} e^{\frac{-|y_j|^2 + (2N|x_j| + 2(N-1)R)|y_j|}{2N\eps}}.
	\end{align*}
	
	When $X$ and $\eps$ are fixed, the latter is an integrable function of the variable $Y = (y_1, \dotsc, y_N)$, and we conclude the first part of the proof thanks to \autoref{generalized-lebesgue}.
	
	Recalling \eqref{nabla-j-definition} we have
	\[ \nabla_{x_j} \Theta^\eps[\mu^n](X) = \frac{\nabla \proj{\mu^n}{j}(x_j)}{\proj{\mu^n}{j}(x_j)} \Theta^\eps[\mu^n](X) - \int \frac{\nabla \eta^\eps(y_j-x_j)}{\eta^\eps(y_j-x_j)} P^\eps[\mu^n](X,Y) \dY \]
	and
	\[ \nabla_{x_j} \Theta^\eps[\mu](X) = \frac{\nabla \proj{\mu}{j}(x_j)}{\proj{\mu}{j}(x_j)} \Theta^\eps[\mu](X) - \int \frac{\nabla \eta^\eps(y_j-x_j)}{\eta^\eps(y_j-x_j)} P^\eps[\mu](X,Y) \dY \]
	
	Using the first part and the additional assumption on the pointwise convergence of the gradients, we immediately see that
	\[
	\frac{\nabla \proj{\mu^n}{j}(x_j)}{\proj{\mu^n}{j}(x_j)} \Theta^\eps[\mu^n](X) \longrightarrow \frac{\nabla \proj{\mu^n}{j}(x_j)}{\proj{\mu^n}{j}(x_j)} \Theta^\eps[\mu^n](X),
	\]
	converges pointwise a.e. on $\R^d \times \R^d$.
	
	As for the second term, like before the integrands converge pointwise a.e., and the domination is obtained using \autoref{upper-lower-lemma} (i) and (ii).
\end{proof}

From \autoref{pointwise-prop}, using some dominations already seen in \autoref{section-regularity}, we obtain the following corollary.

\begin{corollary} \label{H1-corollary} Suppose that $\mu^n \wconv \mu$, with $\proj{\mu^n}{j} \to \proj{\mu}{j}$ in $\prob^{1,p}(\R^d)$ and $\proj{\mu^n}{j} \to \proj{\mu}{j}$ pointwise a.e. on $\R^d$ for every $j = 1, \dotsc, N$. Then $(\Theta^\eps[\mu^n])^\frac{1}{p} \to (\Theta^\eps[\mu])^\frac{1}{p}$ in $L^p\left( \Rdn \right)$.
	
	Assume in addiction that $\nabla \proj{\mu^n}{j} \to \nabla \proj{\mu}{j}$ pointwise a.e. on $\R^d$. Then $(\Theta^\eps[\mu^n])^\frac{1}{p} \to (\Theta^\eps[\mu])^\frac{1}{p}$ in $W^{1,p}\left( \Rdn \right)$.
\end{corollary}

\begin{proof}
	By \autoref{pointwise-prop} we already have pointwise  a.e. convergence of the functions. Using \eqref{gamma-inequality} we get
	\begin{ieee*}{rCl}
		\abs{(\Theta^\eps[\mu^n](X))^\frac{1}{p} - (\Theta^\eps[\mu](X))^\frac{1}{p}}^p & \leq & \abs{\Theta^\eps[\mu^n](X) - \Theta^\eps[\mu](X)} \\
		& \leq & \Theta^\eps[\mu^n](X) + \Theta^\eps[\mu](X).
	\end{ieee*}
	
	The latter converges pointwise to $2 \Theta^\eps[\mu](X)$, and
	\[ \int \Theta^\eps[\mu^n](X) \dX + \int \Theta^\eps[\mu](X) \dX = 2, \]
	which allows to conclude the first part of the proof thanks to \autoref{generalized-lebesgue}.
	
	Using the expression given by \autoref{thm-regularity} and \eqref{gamma-inequality} we have
	\begin{ieee*}{l}
		\int \abs{\nabla_{x_j} (\Theta^\eps[\mu^n])^\frac{1}{p}(X) - \nabla_{x_j} (\Theta^\eps[\mu])^\frac{1}{p}(X)}^p \dX \\
		\leq \frac{1}{p^p} \int \abs{\Theta^\eps[\mu^n](X)^\frac{1-p}{p} \nabla_{x_j} \Theta^\eps[\mu^n](X) - \Theta^\eps[\mu](X)^\frac{1-p}{p} \nabla_{x_j} \Theta^\eps[\mu](X)}^p \dX.
	\end{ieee*}
	
	By \autoref{pointwise-prop} we have pointwise convergence to zero of the integrand. In order to control the gradients we recall \autoref{nabla-domination} and get
	\begin{ieee*}{Cl}
		\IEEEeqnarraymulticol{2}{l}{\abs{\Theta^\eps[\mu^n](X)^\frac{1-p}{p} \nabla_{x_j} \Theta^\eps[\mu^n](X) - \Theta^\eps[\mu](X)^\frac{1-p}{p} \nabla_{x_j} \Theta^\eps[\mu](X)}^p} \\
		\leq & 2^{p-1} \left( \Theta^\eps[\mu^n](X)^{1-p} \abs{\nabla_{x_j} \Theta^\eps[\mu^n](X)}^p + \Theta^\eps[\mu](X)^{1-p} \abs{\nabla_{x_j} \Theta^\eps[\mu](X)}^p \right) \\
		\leq & 4^{p-1} \left(\frac{\abs{\nabla \proj{\mu^n}{j}(x_j)}^p}{\proj{\mu^n}{j}(x_j)^p} \Theta^\eps[\mu](X) + \int \frac{\abs{\nabla \eta^\eps(y_j - x_j)}^p}{\eta^\eps(y_j-x_j)^p} P^\eps[\mu^n](X, Y) \dY \right) \\
		& {}+ 4^{p-1} \left(\frac{\abs{\nabla \proj{\mu}{j}(x_j)}^p}{\proj{\mu}{j}(x_j)^p} \Theta^\eps[\mu](X) + \int \frac{\abs{\nabla \eta^\eps(y_j - x_j)}^p}{\eta^\eps(y_j-x_j)^p} P^\eps[\mu](X, Y) \dY \right) \\
		=\vcentcolon & 4^{p-1} g_n(X) + 4^{p-1} g(X)
	\end{ieee*}
	
	By hypothesis we have that that $g_n \to g$ pointwise a.e. as in the proof of \autoref{pointwise-prop}. Moreover, as already seen above,
	\[ \int g_n(X) = p^p \int \abs{\nabla \left( \proj{\mu^n}{j} \right)^\frac{1}{p}(x_j) }^p \de x_j + \int \frac{\abs{\nabla \eta^\eps(z)}^p}{\eta^\eps(z)^{p-1}} \de z \]
	and 
	\[ \int g(X) = p^p \int \abs{\nabla \left( \proj{\mu}{j} \right)^\frac{1}{p}(x_j) }^p \de x_j + \int \frac{\abs{\nabla \eta^\eps(z)}^p}{\eta^\eps(z)^{p-1}} \de z, \]
	which allows to conclude thanks to \autoref{generalized-lebesgue}.
\end{proof}

As a final result we obtain \autoref{mu-continuity}, which we report here for the sake of the reader.

\mucontinuity*

\begin{proof}
	By contradiction, suppose that there exist $\delta > 0$ and a subsequence of $(\mu^n)$ (denoted again $(\mu^n)$ for simplicity) such that
	\begin{equation} \label{not-converging}
	d^{1,p}\left( \Theta^\eps[\mu^n], \Theta^\eps[\mu] \right) \geq \delta.
	\end{equation}
	
	Extract a further subsequence $(\mu^{n_k})_k$ such that $\proj{\mu^{n_k}}{j} \to \proj{\mu}{j}$ in $\prob^{1,p}(\R^d)$, and in addition  $\proj{\mu^{n_k}}{j} \to \proj{\mu}{j}$ and $\nabla \proj{\mu^{n_k}}{j} \to \nabla \proj{\mu}{j}$ pointwise a.e. on $\R^d$ for every $j = 1, \dotsc, N$. Due to \autoref{H1-corollary} we should have $(\Theta^\eps[\mu^{n_k}])^\frac{1}{p} \to (\Theta^\eps[\mu])^\frac{1}{p}$ in $W^{1,p}\left( \Rdn \right)$, contradicting \eqref{not-converging}.
\end{proof}


\bibliographystyle{plain}

\nocite{ambrosio2013user, pass2015multi, bindini2017optimal, cotar2013density, cotar2018smoothing}
\bibliography{../biblio}

\end{document}